\documentclass[10pt]{article}
\usepackage{amsmath,amsfonts,amssymb,amsthm} 
\usepackage[latin1]{inputenc}
\usepackage[T1]{fontenc}
\usepackage{bbm}
\usepackage{hyperref}
\usepackage{amsmath}
\usepackage{amsthm}
\usepackage{latexsym}
\usepackage{enumerate}
\usepackage{amssymb}
\usepackage[all]{xy}
\usepackage{calc}
\usepackage{multicol}
\usepackage{bbm}
\usepackage{appendix}

\newtheorem{thm}{Theorem}[section]
\newtheorem{prop}[thm]{Proposition}

\newtheorem{defi}[thm]{Definition}
\newtheorem{lem}[thm]{Lemma}
\newtheorem{cor}[thm]{Corollary}
\newtheorem{rmq}[thm]{Remark}

\numberwithin{equation}{section}

\newcommand{\qtext}[1]{\quad\mbox{#1}\quad} 
\newcommand{\qqtext}[1]{\qquad\mbox{#1}\qquad} 

\newcommand{\sfd}{\mathsf{d}}

\DeclareMathOperator{\supp}{supp}

\newcommand{\E}{{\mathbb E }}  
\newcommand{\eps}{{\varepsilon}}

\newcommand{\field}[1]{\mathbb{#1}}
\newcommand{\N}{\field{N}}          		
\newcommand{\Z}{\field{Z}}          		
\newcommand{\R}{\field{R}}          		

\newcommand{\flr}[1]{\left\lfloor #1 \right\rfloor}

\newcommand{\xto}[1]{\xrightarrow[#1]{}}

\begin{document}

\title{Invariant measures and global well-posedness for a fractional Schr\" odinger equation with Moser-Trudinger type nonlinearity}
\author{Jean-Baptiste Casteras and L\'eonard Monsaingeon}
\maketitle

\begin{abstract}
In this paper, we construct invariant measures and global-in-time solutions for a fractional Schr\" odinger equation with a Moser-Trudinger type nonlinearity
\begin{equation}
\label{eq:E_abstract}
\tag{E}
i\partial_t u= (-\Delta)^{\alpha}u+ 2\beta u e^{\beta |u|^2}
\qqtext{for}(x,t)\in \ M\times \R
\end{equation}
on a compact Riemannian manifold $M$ without boundary of dimension $d\geq 2$.
To do so, we use the so-called Inviscid-Infinite-dimensional limits introduced by Sy ('19) and Sy and Yu ('21).
More precisely, we show that if $s>d/2$ or if $s\leq d/2$ and $s\leq 1+\alpha$, there exists an invariant measure $\mu^{s}$ and a set $\Sigma^s \subset H^s$ containing arbitrarily large data such that $\mu^{s}(\Sigma^s ) =1$ and that \eqref{eq:E_abstract} is globally well-posed on $\Sigma^{s}$.
For strong regularities $s>d/2$ we also obtain a logarithmic upper bound on the growth of the $H^r$-norm of our solutions for $r<s$.
This gives new examples of invariant measures supported in highly regular spaces in comparison with the Gibbs measure constructed by Robert ('21) for the same equation. 
\end{abstract}

\maketitle



\section{Introduction}

In this paper, we are interested in the construction of invariant measures and global-in-time solutions for the following fractional Schr\" odinger equation with defocusing Moser-Trudinger type nonlinearity 
\begin{equation}
\label{eq}
i\partial_t u= (-\Delta)^{\alpha}u+ 2\beta u e^{\beta |u|^2}
\qqtext{for}(x,t)\in \ M\times \R.
\end{equation}
Here $\alpha>0$ and $\beta >0$ are fixed parameters, and $M$ is a compact Riemannian manifold without boundary of dimension $d\geq 1$. 

This equation has been introduced in \cite{LamLippmann} for $\alpha=1$ to describe a self-focusing laser beam whose radius is much greater than the vacuum wavelength.
Equation \eqref{eq} is also interesting from a  purely mathematical point of view since when $d=2$ the nonlinearity is energy critical, where the energy is
\begin{equation}
 \label{eq:def_energy_E}
E(u)= \dfrac{1}{2} \|u\|_{\dot{H}^{\alpha}}^2 + \int_M e^{\beta|u|^2}\,dx.
\end{equation}

To illustrate this point, let us consider the fractional Schr\" odinger equation with pure power nonlinearity
\begin{equation}
 i\partial_t u =(-\Delta)^{\alpha} u+ |u|^{p-1}u.
\label{eq:NLS_p}
 \end{equation}
Recall that it satisfies two conservation laws, in the sense that the mass and the $p$-energy
$$
M(u)= \dfrac{1}{2} \|u\|_{L^2}^2
\qqtext{and}
E_p(u)= \dfrac{1}{2} \|u\|_{\dot{H}^{\alpha}}^2 + \dfrac{1}{p+1}\|u\|_{L^{p+1}}^{p+1}
$$
are preserved along the evolution, at least formally.
Moreover, \eqref{eq:NLS_p} enjoys a scaling invariance
$$
u_\lambda(x,t)=\lambda^{\frac 1{p-1}} u(\lambda^{\frac 1{2\alpha}} x , \lambda t ),
\qqtext{for}\lambda>0
$$
and therefore the critical $\dot H^{s_c}$ Sobolev regularity is given by the critical exponent
$$
s_c= \dfrac{d}{2} - \dfrac{2\alpha}{p-1}.
$$
The power-NLS is said energy super-critical (resp. energy critical, resp. energy subcritical) if $s_c >\alpha $ (resp. $s_c = \alpha$, resp. $s_c <\alpha$).
Viewing formally the exponential nonlinearity as an infinite-degree polynomial $e^{\beta |u|^2}u= \sum \dfrac{\beta^k}{k!} |u|^{2k} u$ we see that the critical Sobolev scale for \eqref{eq} is given by $p=\infty$, namely
$$
s_c = \dfrac{d}{2}.
$$
As a consequence \eqref{eq} is energy critical in dimension $d=2$ for the standard diffusion $\alpha=1$.
When $M=\R^2$, this equation has been intensively studied see \cite{MR2568809,MR2929605}.
We also refer to \cite{MR559676,MR3360676,MR2281189} for related equations with the same type of nonlinearity.
In particular, the global well-posedness of \eqref{eq} for small initial data and defocusing nonlinearity has been obtained in \cite{MR2568809} when $M=\R^2$ and $\alpha=1$.
In the same setting, a scattering result for small data  when $\beta=4\pi$ was proved in \cite{MR2929605}. 
 On the other hand, when $M$ is a compact manifold, it is interesting to investigate the behavior of \eqref{eq} since scattering should not be expected.
 In fact, one expects that generic solutions should exhibit some energy cascade from low to high frequencies, which should presumably be detected by the unboundedness $\limsup_{|t|\rightarrow \infty} \|u(t)\|_{H^s}=+\infty$ for large enough $s$. 
 Our main goal in this paper is to construct some invariant measures, prove probabilistic global well-posedness for \eqref{eq}, and obtain some information about the long-time behavior of solutions (or rather, the possible growth rate of their Sobolev norms).

The first result in this direction was obtained by Bourgain \cite{MR1309539} for a nonlinear Schr\"odinger equation in dimension $d=1$.
The main obstruction to obtaining global solutions at this low level of regularity is the lack of conservation laws.
In order to tackle this issue, Bourgain was able to use the Gibbs measure constructed by Lebowitz, Rose, and Speer \cite{MR939505} (supported on $H^{1/2-}$) in order to obtain bounds for some $N$-dimensional Galerkin projections of the equation (which were already used in the construction of the Gibbs measure).
Most notably, these bounds were independent of the dimension $N$.
Comparing the full problem with its finite-dimensional projections, Bourgain then managed to prove the global existence of solutions.
To some extent, invariant measures allow conservation laws to survive even in low regularity, at least in a statistical sense.

Concerning our specific equation, Gibbs measures were constructed by Robert in \cite{Robert}.
More precisely, let us set up some notations.
We denote the potential
$$
V_\beta(u) =\int_M e^{\beta |u|^2}dx
$$
and the reference Gaussian measure is (at least formally)
\begin{equation}
\label{eq:Gaussian_field}
``\,\mu_\alpha (du) = \prod\limits_{n\geq 0} \exp\left(- \dfrac12 \left(1+\sqrt{|\lambda_n|}\right)^{2\alpha} |\hat{u}_n|^2\right) d\hat{u}_n\,",
\end{equation}
where $\hat{u}_n =\langle u,e_n\rangle$ is the $n$-th Fourier coefficient in the $L^2$ expansion of $u$ on the Hilbert basis $\{e_n\}_{n\in\N}$ given by the Laplace-Beltrami eigenfunctions with eigenvalues $-\lambda_n$.
Notice that $\mu_\alpha$ is the law of the $L^2$-valued random variable
$$
u_\alpha^\omega= \sum_{n\geq 0} \dfrac{g_n(\omega)}{(1+|\lambda_n|)^{\alpha} }e_n,
$$
where $g_n$ is a family of independent standard complex-valued Gaussians on some probability space $(\Omega ,\mathbb{P})$.
In \cite{Robert} it is shown that, if $\alpha >d/2$ then
$$
\rho_{\alpha ,\beta} (du)= \dfrac{e^{-V_\beta (u)}}{\int_{L^2} e^{-V_\beta (u) } \mu_\alpha (du) } \mu_\alpha (du)
$$
 is a well-defined probability measure supported on $H^{(\alpha-\frac d2)^-}(M)$ and is absolutely continuous with respect to $\mu_\alpha$.
 Let us point out that the condition $\alpha >d/2$ is really necessary for the existence of the Gibbs measure, and that the support $\supp \rho_{\alpha ,\beta}$ necessarily corresponds to rough regularity.

Robert then uses this Gibbs measure in order to construct global solutions to \eqref{eq}:
for fixed $\alpha >d/2$, small enough $0<\beta <\beta_0$, and any $0<s<\alpha-\frac d2$, there exists a random variable $u$ with values in $C(\R ; H^s)$ solving \eqref{eq} in the distributional sense, and $\rho_{\alpha ,\beta}$ remains invariant over time.
In the stronger dispersion regime $\alpha >d$, one can in fact retrieve more information and prove that, for any $0<\beta \ll \beta_0$, the evolution \eqref{eq} is $\rho_{\alpha ,\beta}$-almost surely \emph{globally} well-posed and that the Gibbs measure $\rho_{\alpha ,\beta}$ is flow-invariant.
More precisely, for fixed $\alpha>d$ and any $s\in(\frac{d}{2},\alpha -\frac d2)$, the flow is always locally well-posed and there exists a set $\Sigma^s \subset H^s$ of full $\rho_{\alpha ,\beta}$-measure such that, for any $u_0 \in \Sigma^s$, the flow $\phi^t (u_0)$ is globally defined and $\rho_{\alpha ,\beta}$ is invariant under $\phi^t$.
Invariance is understood here in the sense that, for any $\rho_{\alpha ,\beta}$-measurable set $A\subset \Sigma^s$, there holds $\rho_{\alpha ,\beta}(\phi^t(A))=\rho_{\alpha ,\beta}(A)$ for any $t\in \R$.
For weak dispersions $d/2<\alpha \leq d$, let us stress that even the \emph{local} well-posedness is problematic (in particular uniqueness poses serious issues).

An alternative approach to the construction of Gibbs measures was developed by Kuksin \cite{MR2070104} for the $2$-d Euler equation, and adapted by Kuksin and Shirikyan \cite{MR2039838} for the cubic Schr\" odinger equation.
This is known as the \emph{fluctuation-dissipation method}, and consists in considering a stochastic perturbation of the original equation, enjoying suitable enhanced dissipation properties and possessing a stationary measure for any given viscosity parameter.
Using a compactness argument, one wishes to show that this family of invariant measures has a weak limit as the viscosity vanishes, and that this limit is as expected an invariant measure for the original equation.
As a byproduct of the construction, suitable estimates (uniform in the level of noise) can be used next to prove long-time existence results.
We refer to \cite{MR4271958,MR3443633,MR3812859,latocca2020construction,SY1,SY2} for related results.

Sy \cite{S} and Sy and Yu \cite{SY3} combined these two approaches into what they call \emph{Inviscid-Infinite-dimensional limits}, and constructed invariant measures for a class of energy-supercritical NLS equations with pure-power nonlinearities.
This method consists in using a fluctuation-dissipation argument on $N$-dimensional Galerkin approximations.
In doing so, one obtains first a stationary measure for any fixed viscosity $\sigma>0$ and dimension $N\in\N$.
An invariant measure for the original equation is then retrieved by taking first an inviscid limit $\sigma\to 0$ and then an infinite-dimensional limit $N\to\infty$.
This is exactly the method that we implement in this work, with several technical improvements and difficulties arising here due to the exponential nonlinearity.
\\

Throughout the whole paper we focus on the energy critical and supercritical cases only, i-e $\alpha \leq d/2$. (The subcritical case can be handled by standard fixed point methods.)
In the high regularity setting, $s>d/2$, our main result reads
\begin{thm}
\label{thmstrong}
Let $d\geq 2$, $\alpha\leq d/2$, and fix $s>d/2$.
There exist an increasing concave function $\zeta:\R^+ \rightarrow \R^+$, a probability measure $\mu^{s}\in \mathfrak p(L^2)$ supported on $H^s$ and a set $\Sigma^{s} \subset H^{s}$ such that
\begin{itemize}
\item $\mu^{s} (\Sigma^{s} )=1$
\item equation \eqref{eq} is globally well-posed on $\Sigma^{s}$
\item the induced flow $\phi^t$ leaves the measure $\mu^{s}$ invariant
\item we have
$$
\int_{L^2} \|u\|_{H^s}^2 \mu^{s} (du) <\infty. 
$$
\item 
the set $\Sigma^{s}$ contains data of arbitrarily large size, $\mu^{s} (\{\|u\|_{H^s}>K \})>0$ for all $K>0$
\item 
for any fixed $r<s$ and all $u_0\in\Sigma^s$ we have
$$
\|\phi_t (u_0)\|_{H^{r}} \leq C \zeta (1+\ln (1+|t|))
\qquad\text{for all }
t\in \R,
$$
for some $C=C(u_0,r)$.
 \end{itemize}
\end{thm}
\noindent
Upon closer inspection, our construction actually yields a rough but explicit upper bound of the form $\zeta(z)\sim C\sqrt{ |z|}$ for large $|z|$ and some $C$ depending on the various parameters.
(See the proof of Proposition~\ref{prop:small_measure_growth} for details.)
\\

In the low regularity regime $s\leq d/2$ we will establish
\begin{thm}
\label{thmweak}
Let $d\geq 2$, $\alpha\leq 1$, and fix $s\leq \min \{d/2, 1+\alpha\}$.
There exists a probability measure $\mu^{s}\in\mathfrak p(L^2)$ supported on $H^s$, a probability space $(\Omega,\mathcal F,\mathbb P)$, and a $L^2_{loc}(\R;H^s)\cap C(\R;H^{s-\alpha})$-valued stochastic process $u$ such that
\begin{itemize}
\item $u=u^\omega$ is a global distributional solution of \eqref{eq} for $\mathbb P$-a.a. $\omega\in\Omega$;
\item $\mu^{s}$ is invariant in the sense that the law of $u(t)$ is $\mu^{s}$ for all $t\in \R$
\item We have
$$
\int_{L^2} \|v\|_{H^s}^2 \mu^{s} (dv) <\infty. 
$$
\item $\supp\,\mu^{s}$ contains data of arbitrarily large size, $\mu^{s} (\{\|u\|_{H^s}>K \})>0$ for all $K>0$.
 \end{itemize}
\end{thm}
\begin{rmq}
The $C(\R;H^{r-\alpha})$ continuity of a.a. trajectory $u^\omega$ appears here for technical convenience, and it seems plausible that trajectories should in fact be continuous for the $H^s$ topology.
\end{rmq}

Although the statement of Theorem~\ref{thmweak} is very similar in spirit with Theorem~\ref{thmstrong}, it is worth pointing out that the very notion of invariance is slightly different and that the rigorous statement is more probabilistic in nature than for high regularity.
Indeed, for $s>d/2$ the global solutions in Theorem~\ref{thmstrong} will be somehow constructed by hand, and  the induced flow $u(t)=\phi^t (u_0)$ is deterministically well-posed at least for small times.
However, at the very low level of $H^s$ regularity involved in Theorem~\ref{thmweak} above, even uniqueness of solutions to \eqref{eq} may fail (in addition to existence), so the very notion of local flow is delicate.
The stochastic setting here somehow allows for an implicit selection of solutions (arising from a passage to the infinite-dimensional limit $N\to\infty$ in our construction), hence the statement does not even mention flow-maps and must be weakened in order to tackle this issue.
Apart from that, the only but major difference compared with Theorem~\ref{thmstrong} is that we lose quantitative control over the $H^r$ growth of solutions.

Observe that Theorem~\ref{thmstrong} holds in the whole supercritical range $\alpha\in (0,d/2]$.
This is no longer the case for Theorem~\ref{thmweak}, since the proof relies at some point on a C\'ordoba-C\'ordoba type lemma crucially requiring $\alpha \leq 1$.
This restriction seems artificial, as one expects that higher dispersion should simplify the analysis, and indeed this is a purely technical obstacle due to our specific construction.
We point out that, also for technical reasons, we were not able to cover the full range of Sobolev exponents in the low regularity regime, and $s\in (1+\alpha,d/2]$ is not covered in Theorem~\ref{thmweak} if $d/2>1+\alpha$.
The very same gap also appears in Sy's work \cite{SY1} for the same reasons.

As far as optimality is concerned, let us stress that (i) for the unregularized equation \eqref{eq} there are a priori only two conserved quantities $E(u),M(u)$, (ii) given the choice of a dissipation operator $\mathcal L=\mathcal L_s$ involved in our construction by fluctuation-dissipation (to be described shortly), the best one can hope to control is the variation $\mathcal E,\mathcal M$ of $E,M$ in the direction of $\mathcal L_s$, and (iii) our current choices of $\mathcal L_s$ control at best the $H^s$ norm, and no better.
This strongly suggests that our measures constructed in Theorem~\ref{thmstrong} and Theorem~\ref{thmweak} should be somehow optimal, in the sense that they should satisfy $\int_{L^2} \|u\|_{H^r}^2 \mu^{s} (du)=\infty $ for any $r>s$ and should thus be $H^s$ critical for given $s$.
(We will not discuss any further this delicate optimality issue in this work.)

Finally, we wish to emphasize that our results allow for arbitrary exponential rates $\beta>0$ in the nonlinearity, contrarily to \cite{Robert}.
It is also worth stressing that we restrict here to defocusing nonlinearities: Indeed, a cornerstone of our analysis will consist in retrieving good statistical control on some suitable dissipation functionals $\mathcal E,\mathcal M$ (to be defined later), thus allowing conservation laws to survive in low regularity and yielding sufficient compactness throughout the successive inviscid-infinite-dimensional limits.
In the focusing case these functionals fail to be coercive, and the whole analysis collapses.

In this work we shall only address the existence of invariant measures and the associated almost-sure global well-posedness.
An interesting and more difficult question is to derive qualitative properties of these measures, e.g. the dimensionality of the supports and the absolute continuity with respect to the Gaussian measure \eqref{eq:Gaussian_field}.
These issues are left for future work.

\subsection*{Structure of the proof}
Let us now describe informally our Inviscid-Infinite-dimensional construction, based on Sy \cite{S} and Sy and Yu \cite{SY3}. 
Let $E_N$ be the space generated by the first $N$ eigenvalues of the Laplace-Beltrami operator on $M$, and $P_N$ the orthogonal projector thereupon.
On $E_N$, we define a Brownian motion and white noise by
$$
\zeta_N (t,x)= \sum_{|m|\leq N} a_m \beta_m (t) e_m (x)
\qqtext{and}\eta_N (t,x)= \dfrac{d}{dt} \zeta_N (t,x),
$$
where $(a_m)_{m\in \Z}$ is an arbitrary family of complex numbers (to be well-chosen later on and satisfying some decay conditions) and $\beta_m$ is a sequence of independent standard real Brownian motions.
 We consider the following (finite-dimensional) SDE
\begin{equation}
\label{eqstointro}
\partial_t u =- i \left[(-\Delta)^{\alpha}  u + P_N \left(2 \beta   u e^{\beta |u|^2} \right)\right] -\sigma^2 \mathcal{L}(u)+\sigma \eta_N,
\end{equation}
where $\sigma>0$ is a small viscosity parameter.
Here $\mathcal L$ is a (deterministic) dissipation operator to be well-chosen below, and the Brownian motion of course induces stochastic fluctuations in the system.
Clearly as $\sigma\to 0$ and $N\to\infty$ this is a perturbation of \eqref{eq}, and note that scaling $\sigma^2\mathcal L$ and $\sigma \eta_N$ is consistent with It\^o's formula.
For technical reasons we will choose different dissipation operators $\mathcal L$, tailored to either strong or low regularity regimes.
In any case the choice of $\mathcal L=\mathcal L_s$ will be subordinated to the prior choice of a fixed Sobolev exponent $s>0$, thus targeting a chosen $H^s$ regularity.

We will prove that \eqref{eqstointro} is always stochastically well-posed on $E_N$ (see Definition \ref{def:stochastic_wellposedness}), and standard finite-dimensional Bogoliubov-Krylov arguments will allow to produce a stationary measure $\mu^s_{\sigma ,N}$ on $E_N$.
The energy and mass dissipation
 $$
 \mathcal E(u)=E'(u;\mathcal L(u))
 \qqtext{and}
 \mathcal M(u)=M'(u;\mathcal L(u))
 $$
 in the direction of $\mathcal L$ will be our two fundamental functionals, yielding barely sufficient compactness and $H^s$ control throughout the whole analysis.
 More precisely, It\^o's formula will yield
 $$
 \int_{L^2} \mathcal{E} (u) \mu^s_{\sigma ,N} (du)\leq C
 \qqtext{and}
 \int_{L^2} \mathcal{M} (u) \mu^s_{\sigma ,N} (du)\leq C
 $$
uniformly in $\sigma,N$.
Owing to our specific choices of $\mathcal L$ (for both strong and low regularities), the resulting functionals $\mathcal E,\mathcal M$ will enjoy suitable coercivity and give just enough compactness on the family of stationary measures $\left\{\mu^s_{\sigma ,N}\right\}_{\sigma,N}$.
Taking first the limit $\sigma \rightarrow 0$ and then $N\rightarrow \infty$, we establish the weak convergence of $\mu^s_{\sigma,N}$ (in $H^r$ for all $r<s$) to a measure $\mu^s$ satisfying the same estimates.
Fine-tuning the Brownian coefficients $(a_m)_m$, one can moreover slightly modify $\mu^s$ in order to get initial data with arbitrarily large $H^s$ norm, which guarantees in particular that $\mu^s$ is not trivially concentrated on the zero solution.

At this stage it remains to construct global solutions for a large-enough class of initial data $u_0\in \Sigma^s$ (with $\mu^s(\Sigma^s)=1$), and this is when the construction starts to really differ for strong and weak regularities.

For high regularities $s>d/2$ we choose to include a strong damping term $\mathcal L(u)=(\dots)+G(\|u\|_{H^s})u$ in the dissipation operator, for some function $G:\R^+\to\R^+$ growing sufficiently fast in order to overtake the exponential nonlinearity.
The Sobolev embedding $H^s\subset L^\infty$ then allows to control the $H^s$ norms of the projected solutions uniformly in $N,\sigma$.
This puts us in position of implementing Bourgain's strategy \cite{MR1309539} to construct a full $\mu^s$-measure set $\Sigma^s \subset H^s$ on which \eqref{eq} is globally well-posed.
The invariance of $\mu^s$ will naturally follow from some strong convergence of the projected flow-map $\phi^N_t$ towards its unprojected counterpart $\phi^t$.
Let us point out that this construction also provides a quantitative estimate on the slow logarithmic growth of the $H^r$-norm of the global solutions constructed here, inherited from the control for finite $N$.

For low regularities $s\leq d/2$ the failure of the Sobolev embedding $H^s\not\subset L^\infty$ prevents from controlling the nonlinearity in $H^s$, thus one cannot expect to control the growth of solutions even for the projected equation.
In order to circumvent this issue we choose now to include the nonlinearity $\mathcal L(u)=(\dots) + P_N\left(u e^{\beta |u|^2}\right)$ in the dissipation operator.
This will result in coercivity $\mathcal M(u)\gtrsim \int |u|^2e^{\beta |u|^2}$ in the mass dissipation functional, which in turn will allow a statistical control $\int_{L^2} \left\|u e^{\beta|u|^2}\right\|_{L^1_{t,x}}\mu^{s}_{\sigma,N}(du)$ for the nonlinearity in any finite time interval.
However, we only managed to control some residual crossed terms via a suitable Cord\'oba-Cord\'oba inequality requiring $s\leq \alpha+1$ as well as2 $\alpha\leq 1$, and our analysis unfortunately does not cover the range $s\in (\alpha+1,d/2]$ (we were simply not able to find a dissipator $\mathcal L$ allowing to control the nonlinearity while dominating the crossed terms.)
At this stage, a soft but very flexible Skorokhod-type argument by Burq, Thomann and Tzvetkov \cite{BTT} allows to conclude that $\sigma^s=\lim \sigma^s_{\sigma,N}$ is invariant, as the limit of invariant measures for which the nonlinearities $P_N(u_Ne^{\beta|u_N|^2})\to ue^{\beta |u|^2}$ enjoy sufficient convergence (precisely owing to Skorokhod's representation theorem).
We also refer to \cite{albeverio1990global,da2002two} for earlier implementations of this strategy in a fluid-mechanical context.
We stress that, for such weak regularities $s\leq d/2$, we do not obtain quantitative control on the growth of $H^r$-norm of the solutions.
As already pointed out, this is the only significant difference between Theorem~\ref{thmweak} and Theorem~\ref{thmstrong}.
\\

The plan of this paper is the following: 
in Section $2$ we establish several basic properties satisfied by solutions of finite-dimensional Galerkin projections of \eqref{eq}.
In particular, we prove existence of a unique global solution in $C(\R;H^s)$.
Moreover, for strong regularities $s>d/2$ we derive suitable quantitative bounds on the $H^r$ norms in fixed time interval, uniformly in $N$.
In the same regime, we also obtain some continuity properties of the flow-map $\phi^t$ associated with \eqref{eq} and of its projected approximation.

Section $3$ focuses on the strong regularity $s>d/2$ and actually contains the core of the paper: we carry out the full technical details for the above construction and prove Theorem \ref{thmstrong}.
This section is the most involved of the paper, as a significant part will be recycled in the next section (mostly the construction of the invariant measure by compactness arguments).

Section $4$ is devoted to the low regular setting, Theorem \ref{thmweak}.
As most of the technical details are similar to the previous section, we will not give the full details for the construction of the measure $\mu^s$.
We will of course highlight the differences and give a full proof for the significantly different construction of global solutions as well as for the invariance of the measure.

Finally, some technical and auxiliary results are collected in the Appendix.

\section{Preliminaries and Galerkin projections}
Let $0=\lambda_0 \leq \lambda_1 \leq \ldots \leq \lambda_m \leq \ldots$ be the eigenvalues of the Laplace-Beltrami operator $-\Delta$ on $M$.
For notational convenience we write $\lambda_{-m}=\lambda_m$.
We denote by $(e_m)_{m\in \N}$ the corresponding eigenfunctions and by $e_{-m}$ the eigenfunction $i e_m$, so that $(e_m)_{m\in \Z}$ is a Hilbert basis for $L^2 (M ;\mathbb{C})$.

We denote by $\|.\|$ the standard $L^2$ norm with corresponding scalar product
$$
(u,v)= \mathcal{R} \int_M u \bar{v}\, dx, 
$$
where $\mathcal{R}$ stands for the real part.
For $u=\sum u_m e_m$ expressed in this Hilbert basis, the standard Sobolev norm reads
$$
\|u\|_{H^s}^2=\|u\|^2+\|u\|_{\dot H^s}^2=\sum\limits_{m\in\Z}(1+\lambda_{m}^s)|u_m|^2,
$$
Let now $E_N$ be the subspace of $L^2$ spanned by $(e_m)_{|m|\leq N}$, and let $P_N$ be the orthogonal projection onto $E_N$.
 On the finite-dimensional space $E_N$, we have the norm equivalence
$$
c\|u\| \leq \|u\|_{L^q}\leq  C_{N,q} \|u\|
$$
and
$$
\|u\|\leq \|u\|_{H^s} \leq \left(1+\lambda_N^{s/2}\right) \|u\| 
$$
for any $q\in (2,\infty]$ and $s\in (0,\infty)$.

For $\gamma>0$ the spectral powers of the Laplacian are defined as
$$
(-\Delta)^\gamma u =\sum \lambda_m^\gamma u_m e_m
\qqtext{for}
u=\sum u_m e_m,
$$
and we recall that $P_N$ and $(-\Delta)^\gamma$ commute.
On $L^2(M)$ we define the dispersive propagator $S(t)$ by
$$
S(t) u =e^{-it(-\Delta)^\alpha }u=\sum e^{-it\lambda_m^\alpha}u_me_m,
$$
and we recall that it is unitary in the sense that
$$
\|S(t)u\|_{H^s}=\|u\|_{H^s}
\qqtext{for all}s\geq 0,\,u\in H^s.
$$

In this setting, we consider now the Galerkin projection of \eqref{eq}
\begin{equation}
\label{eqproj}
\begin{cases}
i\partial_t u =(-\Delta)^{\alpha} u + P_N \left(2 \beta   u e^{\beta |u|^2} \right),\\
 u(0)=P_N u_0,
\end{cases}
\end{equation}
which is nothing but a nonlinear ODE on the finite-dimensional space $E_N$ (note that $(-\Delta)^{\alpha}$ maps $E_N$ to itself).

Let us start with various technical preliminaries on the well-posedness of \eqref{eq} and \eqref{eqproj}.
For fixed $T>0$ and $r\geq 0 $ we write
$$
X_T^r =C ([-T,T]; H^r(M))
\qqtext{with}
\|u\|_{X^r_T}=\sup\limits_{t\in [-T,T]} \|u(t)\|_{H^r}.
$$

\begin{prop}
\label{prop:exbigsN}
Fix $s>d/2$.
For any $R>0$, there exists a time $T>0$ such that, for every $N\in \N$ and $u_0 \in B_R (H^s)$, there exists a unique solution $u \in X_T^s$ to \eqref{eqproj}.
This time can be taken of the form
\begin{equation}
T(s,R)=\frac{c}{\beta e^{C\beta R^2}}
\label{eq:def_existence_time_T}
\end{equation}
for some small $c$ and large $C$ depending on $s$ only.
Moreover, this solution is in fact global-in-time and satisfies 
\begin{equation}
\label{eq:control_uN_2R}
\|u\|_{X_T^s} = \sup_{t\in [-T ,T]} \|u(t)\|_{H^s} \leq 2 R.
\end{equation} 
\end{prop}
The important point here is that we have a quantitative control \eqref{eq:control_uN_2R} on the $H^s$ norm in any fixed time interval uniformly in $N$.

\begin{proof}
Let $u_0 \in B_R (H^s )$.
For $u\in X^s_T$ we set
$$
F(u)= S(t) P_N u_0 - i \int_0^t S (t-\tau)  P_N \left(2 \beta  u e^{\beta |u|^2} \right) d\tau,
$$
and observe of course that $u$ solves \eqref{eqproj} with initial datum $P_Nu_0$ if and only if it is a fixed point for $F(u)=u$.
We have, using the fact that $S$ is unitary in $H^s$ and Lemma \ref{lemsob},
\begin{align*}
\|F(u)\|_{X_T^s}\leq \|P_N u_0\|_{H^s} +2\beta  \int_0^T \left\|u e^{\beta |u|^2} \right\|_{H^s} d\tau\\
 \leq \|u_0\|_{H^s} +C\beta T\|u\|_{X^s_T} e^{C\beta\|u\|_{X^s_T}^2 }. 
\end{align*}
Taking $T \lesssim \dfrac{c}{ \beta e^{C\beta R^2}}$ for $c$ small enough and $C$ large enough (only depending on $s$), we see that $\|F(u)\|_{X^s_T}\leq \|u_0\|_{H^s}+\frac 12\|u\|_{X^s_T}\leq 2R$ and thus $F$ maps $B_{2R}(X^s_T)$ to itself.

 Let now $u_1,u_2$ be two functions in $B_{2R} (X_T^s)$.
 By Lemma~\ref{lemsob} and because $S(t)$ is unitary on $H^s$, we have
\begin{multline}
\|F(u_1) -F(u_2)\|_{X_T^s} 
=
\left\|2i\beta\int_0^tS(t-\tau)P_N\left(u_1e^{\beta|u_1|^2}-u_2e^{\beta|u_2|^2}\right)d\tau\right\|_{X^T_s}
\\
\leq C\beta\int_0^T\left\|u_1e^{\beta|u_1|^2}-u_2e^{\beta|u_2|^2}\right\|_{X^s_T}d\tau
\\
\leq 
C\beta Te^{C\beta\left(\|u_1\|^2_{X_T^s}+\|u_2\|_{X^s_T}^2\right)}\|u_1-u_2\|_{X^s_T}
\leq C\beta Te^{8C\beta R^2} \|u_1-u_2\|_{X^s_T}.
\label{eq:control_expu-expv_Hs}
\end{multline}
This shows that $F$ is a contraction on $B_{2R}(X_T^s)$, again as soon as $T \approx \dfrac{c}{e^{C\beta R^2}}$ for some small $c>0$ and large $C>0$ depending on $s$ only.
By the Banach fixed point Theorem we conclude that there exists a unique solution $u$ to \eqref{eqproj} in $C([-T ,T] ; B_{2R}(H^s))$.

Finally, at least for $u\in E_N$ there holds $\left(i(-\Delta)^{\alpha} u,u\right)=0$ and $\left( i P_N \left( u e^{\beta |u|^2}\right),u \right)=\left( i  u e^{\beta |u|^2},u \right)=0$, whence the conservation of mass $\dfrac{d}{dt}\|u\|^2=2\left(u,\partial_tu\right)=0$ along solutions of \eqref{eqproj}.
On the finite-dimensional space $E_N$ this guarantees that $u(t)$ remains in a fixed compact set (a ball of radius $\|P_Nu_0\|$), hence blowup in finite time is excluded and $u(t)$ is thus global-in-time.
\end{proof}

In the low regularity setting we have
\begin{prop}
Assume that $s\leq d/2$.
For any $u_0 \in B_R (H^s)$ there exists a unique global solution $u\in C (\R; H^s )$ to \eqref{eqproj}.
\end{prop}
The main and important difference with Proposition~\ref{prop:exbigsN} is that, for such low regularity, we lose quantitative $H^s$ estimates in uniform time intervals such as \eqref{eq:control_uN_2R}.
Some control can still be retrieved in intervals $[-T_N,T_N]$, but with $T_N$ exponentially small as $N\to \infty$.
This is reminiscent from the fact that, for low regularity $s<d/2$, the well-posedness of \eqref{eq} for $N=\infty$ is much more delicate.
\begin{proof}
The proof is quite similar to the previous one, with the major difference that $H^s$ does not embed into $L^\infty$ anymore.
Instead, we use the norm equivalence $\|u \|_{L^\infty} \leq CN^{\frac d2 -s}  \|u\|_{H^s} $ for $u\in E_N$ \cite[Lemma 3.3]{Robert} in order to get similarly
\begin{align*}
\|F(u)\|_{X_T^s} \leq \|u_0\|_{H^s} +C T N^{\frac d2 -s}\|u\|_{H^s} e^{CN^{d-2s} \beta\|u\|_{H^s}^2 }. 
\end{align*}
Taking $T \lesssim \dfrac{c}{N^{d/2 -s} e^{CN^{d-2s}\beta R^2}}$ for suitably chosen constants $c,C>0$, we have $\|F(u)\|_{X_T^s} \leq 2R$ for all $u\in B_{2R} (X_T^s )$ and $F$ thus maps $B_{2R}(X_T^s)$ to itself.
(Note that the upper bound for $T$ is exponentially small as $N\to\infty$.)
The rest of the proof follows similarly.
\end{proof}
Let us now turn to the unprojected equation \eqref{eq} for $N=\infty$.
For $s>d/2$ we have local well-posedness:
\begin{prop}
\label{prop:exbigs}
 Assume that $s>d/2$.
For any $R>0$, there exists a time $T>0$ such that, for every $u_0 \in B_R (H^s)$, there exists a unique solution $u\in X_T^s =C ([-T,T]; H^s )$ to \eqref{eq} satisfying moreover
$$
\|u\|_{X_T^s} = \sup_{t\in [-T ,T]} \|u(t)\|_{H^s} \leq 2 R.
$$ 
The time $T=T(s,R)$ can be taken exactly as in \eqref{eq:def_existence_time_T}
\end{prop}
The proof is identical to that of Proposition~\ref{prop:exbigsN} (a fixed point argument heavily relying on Lemma~\ref{lemsob}) and we omit the details.
The important point is here that the local existence time $T=T(s,R)$ is identical for $N<\infty$ and $N=\infty$.
Of course, even in the strong regularity regime $s>d/2$ we do not claim global existence at this stage (since ultimately this is the whole purpose of this paper, for as large a class of initial data as possible).
\\

The rest of this section is devoted to further properties of the flows for $s>d/2$.
(For $s\leq d/2$ even the \emph{local} well-posedness is unclear, in particular uniqueness is an issue.)

We denote by $\phi^t$ and $\phi_N^t$ the corresponding flows associated with \eqref{eq} and \eqref{eqproj} (local and global-in-time, respectively), implicitly well-defined according to Proposition~\ref{prop:exbigsN} and Proposition~\ref{prop:exbigs}.
\begin{lem}
\label{lem:phitN_converges_phit}
Fix $r>d/2$, let $u_0\in B_R(H^r)$ for some $R>0$, and choose $T=T(r,2R)$ as in \eqref{eq:def_existence_time_T}.
If $u_{0,N}\in E_N$ is any sequence such that $u_{0,N}\to u_0$ in $H^r$ then
\begin{equation}
\|\phi^t (u_0) - \phi_N^t (u_{0,N}) \|_{X_T^r} \to 0
\hspace{1cm}\mbox{as }N\to\infty . 
\label{eq:phiNt_to_phit}
\end{equation}
Moreover both flows are locally Lipschitz continuous, in the sense that if $u_0,u'_0\in B_R(H^r)$ and $T=T(r,R)$ as before, then
\begin{equation}
\|\phi^t (u_0) - \phi^t (u'_0) \|_{X_T^r}\leq C \|u_0-u'_0\|_{H^r}
\label{eq:continuity_phi_initial_data}
\end{equation}
and
\begin{equation}
\|\phi_N^t (u_0) - \phi_N^t (u'_0) \|_{X_T^r}\leq C \|u_0-u'_0\|_{H^r}
\label{eq:continuity_phiN_initial_data}
\end{equation}
for some $C=C(R,r)$.
\end{lem}
\begin{proof}
Let us begin with \eqref{eq:phiNt_to_phit}.
Note first that, if $N$ is large enough, then $u_0,u_{0,N}\in B_{2R}$ hence $\phi^t(u_0)$ and $\phi^t_N(u_{0,N})$ are both defined at least for times $|t|\leq T=T(r,2R)$ and both remain in $B_{4R}(H^r)$ (Propositions~\ref{prop:exbigsN} and \ref{prop:exbigs}).
Using Duhamel's formula, we have, for any $t\in [0,T)$
\begin{align*}
\phi^t (u_0) &
= S(t) u_0 - 2i\beta \int_0^t S(t-\tau)   \phi^\tau (u_0) e^{\beta |\phi^\tau (u_0)|^2}  d\tau, 
\\
\phi_N^t (u_{0,N})
&= S(t) u_{0,N} - 2i\beta \int_0^t S(t-\tau)  P_N \left(  \phi_N^\tau (u_{0,N}) e^{\beta \left|\phi_N^\tau (u_{0,N}) \right|^2 } \right) d\tau .
\end{align*}
Taking the difference, we get
\begin{align*}
\phi^t (u_0) - \phi_N^t (u_{0,N})
&= S(t) (u_0 - u_{0,N}) -2 i  \beta \int_0^t S(t-\tau) (1- P_N ) \left(\phi^\tau (u_0)e^{\beta |\phi^\tau (u_0) |^2} \right) d\tau \\
& + 2i  \beta \int_0^t S(t-\tau) P_N \left[
e^{\beta \left| \phi^\tau_N ( u_{0,N})\right|^2}  \phi_N^\tau ( u_{0,N}) - e^{\beta |\phi^\tau (u_0)|^2} \phi^\tau (u_0)
\right]d\tau .
\end{align*}
Using that $S(t)$ is unitary in $H^r$, exploiting Lemma~\ref{lemsob} to control the last integral, and recalling that $\phi^\tau(u_0),\phi^\tau_N(u_{0,N})$ remain contained in a fixed ball $B_{4R}(H^r)$ for times $\tau\in[-T,T]$, we get
\begin{align}
\| \phi^t (u_0) - \phi_N^t (u_{0,N})\|_{H^r}
&\leq  \|u_0-u_{0,N}\|_{H^r} +C \int_0^T  \left\| (1- P_N ) \left[\phi^\tau (u_0)e^{\beta |\phi^\tau (u_0) |^2} \right]\right\|_{H^r} d\tau 
\notag
\\
& + Ce^{16C\beta R^2} \int_0^t    \left\| \phi^\tau (u_0)- \phi_N^\tau ( u_{0,N})  \right\|_{H^r} d\tau .
\label{eq:estimate_PhiN_Phi_gronwall}
\end{align}
Because $f(\tau)=\phi^\tau (u_0)e^{\beta |\phi^\tau (u_0) |^2}\in H^r$ it is easy to see, for fixed $u_0$, that we have pointwise convergence
$$
\left\| (1- P_N ) \left[\phi^\tau (u_0)e^{\beta |\phi^\tau (u_0) |^2} \right]\right\|_{H^r}\xto{N\to\infty}0
\qquad\mbox{for a.e. }\tau\in [-T,T],
$$
as the remainder $\sum_{|m|>N}(\dots)$ of an absolutely convergent series.
The uniform bound $\phi^\tau (u_0)\in B_{4R}(H^r)$ combined with Lemma~\ref{lemsob} allows to conclude by Lebesgue's dominated convergence that
$$
I_N=\int_0^T  \left\| (1- P_N ) \left[\phi^\tau (u_0)e^{\beta |\phi^\tau (u_0) |^2} \right]\right\|_{H^r} d\tau \xto{N\to\infty} 0.
$$
Our previous estimate \eqref{eq:estimate_PhiN_Phi_gronwall} thus reads
$$
\| \phi^t (u_0) - \phi_N^t (u_{0,N})\|_{H^r}
\leq 
\eps_N + C\int_0^t    \left\| \phi_N^\tau ( u_{0,N})- \phi^\tau (u_0) \right\|_{H^r} d\tau 
$$
with $\eps_N=\|u_0-u_{0,N}\|_{H^r} + I_N\to 0$ as $N\to\infty$.
Gr\"onwall's inequality leads to
$$
\| \phi^t (u_0) - \phi_N^t (u_{0,N})\|_{H^r} \leq \eps_N e^{CT}
$$
and \eqref{eq:phiNt_to_phit} follows.

For \eqref{eq:continuity_phi_initial_data} we take the difference
\begin{multline*}
\phi^t (u_0)- \phi^t (u_0')
= S(t) (u_0-u_0')\\
- 2i\beta \int_0^t S(t-\tau)   \left(\phi^\tau (u_0) e^{\beta |\phi^\tau (u_0)|^2}  -\phi^\tau (u_0') e^{\beta |\phi^\tau (u_0')|^2}\right)d\tau.
\end{multline*}
Because we know that $\phi^\tau(u_0),\phi^\tau(u_0')\in B_{2R}(H^r)$ for all $|\tau|\leq T$, and $S(t)$ is unitary, our Lemma~\ref{lemsob} allows to control
$$
\|\phi^t (u_0)- \phi^t (u_0')\|_{H^r}\leq \|u_0- u_0'\|_{H^r} + C\int_0^t\|\phi^\tau (u_0)- \phi^\tau (u_0')\|_{H^r} d\tau
$$
for some $C=C(R,r)$, and our claim immediately follows by Gr\"onwall's inequality.

The proof of \eqref{eq:continuity_phiN_initial_data} follows along the exact same lines and we omit the details.
\end{proof}
For technical purposes we will also need to upgrade the convergence of $\phi^t_N$ towards $\phi^t$ to uniform convergence on compact sets
\begin{cor}
\label{cor:phiN_to_phi_uniform_con_compacts}
 Let $r>d/2$ and fix any $H^r$-compact set $K$.
 There exists $T=T(r,K)$ such that for any fixed $|t|\leq T$ we have uniform convergence
\begin{equation}
\sup\limits_{u_0\in K}\|\phi^t (u_0) - \phi_N^t (u_0) \|_{H^r} \xto{N\to\infty} 0
\label{eq:phiNt_to_phit_uniform_Hs}
\end{equation}
\end{cor}
\begin{proof}

By compactness, we have first $K\subset B_R(H^r)$ for some large enough $R>0$.
Let $T=T(r,R)$ be as in Propositions~\ref{prop:exbigsN}~\ref{prop:exbigs}, so that $\phi^t(u_0),\phi^t_N(u_0)$ are well-defined and remain contained in $B_{2R}(H^r)$ for all $|t|\leq T$.
By the previous Lemma~\ref{lem:phitN_converges_phit} we see that $\phi^t,\phi^t_N:B_{R}(H^r)\to H^r$ are equi-Lipschitz, uniformly in $N$.
Moreover, taking in particular $u_{0,N}=P_Nu_0$ in \eqref{eq:phiNt_to_phit}, we see that $\phi^t_N(u_0)\to \phi^t(u_0)$ pointwise in $u_0\in B_R(H^r)$.
A variant of the Arzel\`a -Ascoli theorem (Lemma~\ref{lem:pointwise_to_uniform_on_compacts} in the Appendix) finally guarantees that the convergence is in fact uniform on $K$.
\end{proof}

\section{High regularity}
\label{sec:strong}
In this section we focus exclusively on the high regularity regime $s>d/2$, and recall that we only consider the energy critical and supercritical cases $\alpha \leq d/2$.
The reader can think of the $H^s$ scale as being fixed once and for all in this section, and we will rather use the $H^r$ notation for any other Sobolev spaces, $r<s$.
Let us first set-up our fluctuation-dissipation framework.

\subsection{Fluctuation-dissipation and a priori estimates}
\label{subsec:FD}
On $E_N$, we define a Brownian motion and a white noise by
\begin{equation}
\zeta_N (t,x)= \sum_{|m|\leq N} a_m \beta_m (t) e_m (x)
\qqtext{and}\eta_N (t,x)= \dfrac{d}{dt} \zeta_N (t,x),
\label{eq:def_Brownian}
\end{equation}
where $(a_m)_{m\in \Z}$ is an arbitrary family of complex numbers (satisfying some decay conditions to be precised soon), $\beta_m$ is a sequence of independent standard real Brownian motions with respect to a filtration $(\mathcal{F}_t )_{t\geq 0}$ and defined on a probability space $(\Omega , \mathcal{F} , \mathbb{P})$.

Since $s>d/2$ and we consider the energy supercritical case only $\alpha\in(0,d/2]$, we have in particular $s>\alpha$.
We define the dissipation operator
\begin{equation}
\label{eq:def_dissipation_L_strong_disp}
\mathcal{L} (u)= \left[ (-\Delta )^{s-\alpha} + G(\|u\|_{H^{s}})\right] u
\end{equation}
for some smooth, nondecreasing function $G:\R^+\to\R^+$
\begin{equation}
\label{condG}
G(\rho )= c \rho^2e^{\Lambda\beta \rho^2} 
\end{equation}
for some small $c>0$ and large $\Lambda>0$.
Here one can think of $\Lambda$ to be as large as needed, in fact we will at times adjust $\Lambda\geq C$ for various constants $C=C(s)$ related to the Sobolev embedding via the exponential constant in Lemma~\ref{lemsob}.

For some small parameter $\sigma>0$ and fixed exponent $\alpha>0$, consider the Stochastic PDE 
\begin{equation}
\label{eqsto}
\partial_t u =- i \left[(-\Delta)^{\alpha}  u + P_N \left(2 \beta   u e^{\beta |u|^2} \right)\right] -\sigma^2 \mathcal{L}(u)+\sigma \eta_N.
\end{equation}
We think of $\mathcal L$ as a deterministic regularizing dissipation (hence the minus sign), and the white noise induces stochastic fluctuations in the system.
As already pointed out, the scaling $\sigma^2\mathcal L$ and $\sigma \eta_N$ is consistent with It\^o's formula and will be crucial in order to retrieve enough compactness estimates and take the limit $\sigma\to 0^+$ later on.

By solutions of the SPDE, we mean here:

\begin{defi}[Stochastic global well-posedness]
\label{def:stochastic_wellposedness}
 We say that \eqref{eqsto} is stochastically globally well-posed on $E_N$ if the following properties hold:

\begin{enumerate}[(i)]
\item
\label{item:stochastic_uniqueness}
for any $E_N$-valued random variable $u_0$ independent of $\mathcal{F}_t$
\begin{itemize}
\item 
for $\mathbb P$-a.a. $\omega$ there exists a solution $u=u^\omega\in C(\R^+ ;E_N)$ of \eqref{eqsto} with initial datum $u_0=u_0^\omega$ in the integral sense, i-e
\begin{equation}
\label{eq:integral_formulation_stochastic_eq}
u(t)=u_0+\int_0^t
\left[
-i\left((-\Delta)^{\alpha }u + P_N\left(2\beta u e^{\beta|u|^2}\right)\right)
-\sigma^2\mathcal L(u)
\right]d\tau
+\sigma \zeta_N(t)
\end{equation}
\noindent
with equality as elements of $C(\R^+;E_N)$.
\item 
if $u_1^\omega ,u^\omega_2 \in C (\R^+; E_N )$ are two such solutions with the same initial data $u^{\omega}_{1,0}=u^{\omega}_{2,0}$ then $u_1^\omega \equiv u_2^\omega$.
\end{itemize}
\item
\label{item:pathwise_continuity_initial_data}
(pathwise continuity w.r.t. initial data)
for $\mathbb P$-a.a. $\omega \in \Omega$, we require that
$$
\lim_{u_0 \rightarrow u_0'} u^\omega (.\,; u_0) = u^\omega (.\,; u_0' )\ \qtext{in}\ C_{\mathrm{ loc}}(\R^+ ;E_N),
$$
where $u_0$ and $u_0'$ are deterministic data in $H^s$ and $u^\omega (.\,; u_0),u^\omega (.\,; u_0')$ are the corresponding solutions of \eqref{eq:integral_formulation_stochastic_eq} for fixed $\omega$
\item
the $E_N$-valued stochastic process $(\omega ,t) \mapsto u^\omega (t) $ is adapted to the filtration $\sigma (u_0 , \mathcal{F}_t).$
\end{enumerate}
\end{defi}
Owing to (\ref{item:stochastic_uniqueness}), any such solution, if any, is unique.
Also, such a process solves the finite-dimensional \emph{forward} It\^o SDE
$$
du =\left(- i \left[(-\Delta)^{\alpha}  u + P_N \left(2 \beta   u e^{\beta |u|^2} \right)\right] -\sigma^2 \mathcal{L}(u)\right)dt +\sigma d\zeta _N
$$
in the classical sense.

\begin{prop}
\label{prop:stochastic_wellposedness}
Equation \eqref{eqsto} is stochastically globally well-posed on $E_N$ in the sense of Definition~\ref{def:stochastic_wellposedness}.
\end{prop}

\begin{proof}
We will construct the solution in the form $u=z+v$, where $z$ will first contain all the Brownian fluctuation and $v$ will then be obtained by solving pathwise a deterministic ODE with coefficients depending on $z$.

To this end let us first introduce
\begin{equation}
z(t)= \sigma \int_0^t e^{-i(t-\tau)(-\Delta)^{\alpha}} d \zeta_N (\tau).
\label{eq:def_z}
\end{equation}
This is well-defined for $\mathbb{P}$-a.a. $\omega \in \Omega$ as a function of time with values in $E_N$, and it is the unique solution to the It\^o SDE
\begin{equation}
\label{defz}
\begin{cases}
dz = - i (-\Delta)^{{\alpha}}z  dt + {\sigma} d\zeta_N ,\\
 z(0)=0.
\end{cases}
\end{equation}
We also have that $z^\omega(.)$ belongs to $C(\R^+ ;E_N)$ for $\mathbb{P}$-a.a. realization.
Because $\|z\|^2$ is a submartingale (being $z$ a martingale), Doob's maximal inequality and It\^o's formula $d\|z\|^2=2\sigma(z,d\zeta_N)+\sigma^2 dt$ (with $(i(-\Delta)^\alpha z ,z)dt=0$) yield moreover, for any $T>0$,
\begin{equation}
\label{boundz}
\left(\E \sup_{t\in [0,T]}\|z(t)\| \right)^2
\leq \E \sup_{t\in [0,T]}\|z(t)\|^2
\leq 2\E\, \|z(T)\|^2
= 2\sigma^2T.
\end{equation}
For any fixed $\omega \in \Omega$ we consider now the deterministic ODE
\begin{equation}
\label{eqsto2}
\begin{cases}
\dfrac{dv}{dt} =- i\left[(-\Delta)^{{\alpha}}v +2  \beta P_N \left(  (v+z) e^{\beta |v+z|^2}\right)\right] 
\\ \hspace{2cm}
- \sigma^2\Big[(-\Delta)^{s-{\alpha}} +G(\|v+z\|_{H^{s}})\Big] (v+z) ,\\
 v(0)= u_0\in E_N,
\end{cases}
\end{equation}
where $u_0=u_0^\omega$ and $z=z^\omega(t)$ are considered as given, deterministic data for fixed $\omega$.
A straightforward application of the Cauchy-Lipschitz theorem guarantees existence and uniqueness of a local-in time solution $v=v^\omega(t)$, and we will show below that the latter is in fact globally defined for $t\in \R^+$.
Also, observe that $u=v+z$ will automatically solve our integral formulation \eqref{eq:integral_formulation_stochastic_eq} for fixed $\omega$.

As a first step, let us prove that, $\mathbb{P}$-almost surely, the solution $v=v^\omega$ of \eqref{eqsto2} is global-in-time.
Using \eqref{eqsto2} as well the conservation laws $(v,i(-\Delta)^\alpha v)=0$ and $\left(u, iP_N(ue^{\beta|u^2|})\right)=\left(u,iue^{\beta|u^2|}\right)=0$ for $u=v+z\in E_N$, we get
\begin{align*}
\dfrac{d}{dt}\left( \frac 12\|v\|^2\right)
& = \left ( v,\dfrac{dv}{dt}\right)
\\
& = \left(v, -2i\beta(v+z) e^{\beta |v+z|^2} \right) 
-\sigma^2 \Big[\|v \|_{\dot H^{s-{\alpha}}}^2 + \|v\|^2 G(\|v+z\|_{H^{s}}) \Big] 
\\
&\hspace{2cm}   -\sigma^2 G(\|v+z\|_{H^{s}} ) (v,z )
\\
&  = -\left(z, -2i\beta(v+z) e^{\beta |v+z|^2} \right) -\sigma^2 \Big[\|v \|_{\dot H^{s-{\alpha}}}^2 + \|v\|^2 G(\|v+z\|_{H^{s}}) \Big] 
\\
&   \hspace{2cm} -\sigma^2 G(\|v+z\|_{H^{s}} ) (v,z )
\end{align*}
Discarding the nonnegative term $\|v\|^2_{\dot H^{s-\alpha}}$ and using repeatedly Young's inequality we control next
\begin{align}
\dfrac{d}{dt}\left( \frac 12\|v\|^2\right)
&   \leq  \dfrac{8\beta^2\|z\|^2}{\sigma^2} + \dfrac{\sigma^2 e^{2\beta \|v+z\|_{L^\infty}^2 } \|v+z\|^2}{2} - \sigma^2  \|v\|^2 G(\|v+z\|_{H^{s}})
\notag
\\
&   \hspace{2cm} + \sigma^2  G(\|v+z\|_{H^{s}} ) \left[\dfrac{ \|v\|^2}{ 2}+ \dfrac{\|z\|^2 }{2} \right]
\notag 
\\
&   \leq  \dfrac{8\beta^2\|z\|^2}{\sigma^2} + \dfrac{\sigma^2 e^{2\beta C(s) \|v+z\|_{H^{s}}^2 } \|v+z\|_{H^{s}}^2}{2} - \sigma^2  \|v\|^2 G(\|v+z\|_{H^{s}})
\notag
\\
&   \hspace{2cm} + \sigma^2  G(\|v+z\|_{H^{s}} ) \left[\dfrac{ \|v\|^2}{ 2}+ \dfrac{\|z\|^2 }{2} \right]
\notag 
\\
&   \leq  \dfrac{8\beta^2\|z\|^2}{\sigma^2} +\dfrac{\sigma^2}{2}G(\|v+z\|_{H^{s}} )  \big[c+\|z\|^2 -\|v\|^2\big],
\label{eq:estimate_dz2dt}
\end{align}
where we made a crucial use of our growth assumption \eqref{condG} with $\Lambda\geq C(s)$ in the last inequality.

Observe that, due to \eqref{boundz}, for $\mathbb{P}$-almost all $\omega \in \Omega$, there exists a constant $C_{\omega,T}$ such that $z=z^\omega$ satisfies some bound
\begin{equation}
\sup_{t\in [0,T]} \|z^\omega (t)\|^2:=Z_{\omega,T}^2= \sigma^2 C_{\omega,T}       .
\label{eq:bound_z_realization}
\end{equation}
For fixed $t\in [0,T]$ and $\omega\in \Omega$, we have in\eqref{eq:estimate_dz2dt} that either $\|v(t)\|^2\leq c+Z_{\omega,T}^2$, or $\|v(t)\|^2\geq c +Z^2_{\omega,T} \geq c+\|z(t)\|^2$.
In this latter case, the last term in \eqref{eq:estimate_dz2dt} is nonpositive and we find that 
$\dfrac{d}{dt}\left( \frac 12\|v\|^2\right)
\leq  \dfrac{8\beta^2Z^2_{\omega,T}}{\sigma^2}
\leq 8\beta^2C_{\omega,T}
$.
Owing to the rough bound $Z_{\omega,T}=\sigma C_{\omega,T}\leq C_{\omega,T}$ for $\sigma\leq 1$, this readily gives an a-priori estimate
\begin{equation}
\|v(t)\|^2
\leq \|v(0)\|^2+(c+C^2_{\omega,T})
+16\beta^2C_{\omega,T}T,
\hspace{1cm}\forall\,t\in [0,T]
\label{eq:a_priori_bound_v}
\end{equation}
for a.a. $\omega\in \Omega$, hence blowup in finite time is excluded and $v$ is indeed almost surely global-in-time.
Note that this pathwise estimate is independent of $\sigma$, which will be important later on when we take $\sigma\to 0$ (in particular for Proposition~\ref{propmainN}).

Next, let us prove uniqueness.
Assume $u_i=u_i^\omega$, $i=1,2$ are two solutions of \eqref{eq:integral_formulation_stochastic_eq} with deterministic initial data $u_{0,i}$.
For $z=z^\omega(t)$ as in \eqref{eq:def_z}\eqref{defz} and a.a. $\omega\in \Omega$ fixed, note that $v_i:=u_i-z$ both satisfy the deterministic ODE \eqref{eqsto2} with respective initial data $v_{i,0}=u_{i,0}$.
In particular, for given $\omega$, the a priori bounds \eqref{eq:a_priori_bound_v} hold and $u_i=v_i+z$ can be bounded as
$$
\sup\limits_{t\in[0,T]}\|u_i^\omega(t)\| \leq C_{\omega,T}
$$
for a common constant $C_{\omega,T}$ possibly depending on $\omega$, $T$, and $\|u_{i,0}\|$ only.
Moreover, setting
$$
F(u)=- i \left[(-\Delta)^{\alpha}  u + P_N \left(2\beta   u e^{\beta |u|^2} \right)\right]
-\sigma^2  \big[ (-\Delta )^{s-\alpha} + G(\|u\|_{H^{s}})\big] u
$$
we see that $w=u_1-u_2$ satisfies, again for fixed $\omega$, the deterministic ODE
$$
\partial_t w = F(u_1) -F(u_2).
$$
Because $u_i^\omega(t)$ are bounded a priori and $F$ is locally Lipschitz on the finite-dimensional space $E_N$, we have
$$
\|w(t)\|\leq \|w(0)\|+ L_{\omega,T}\int_0^t \|w(\tau)\|d\tau
,\hspace{1.5cm}\forall\,t\in[0,T]
$$
for some local Lipschitz constant $L_{\omega,T}$ depending only on the above upper bound $C_{\omega,T}$ on $\|u_i(\omega,t)\|$.
Gr\"onwall's inequality gives then
$$
\sup\limits_{t\in[0,T]}\|u_1^\omega(t)-u_2^\omega(t)\|\leq \|u_{1,0}^\omega-u_{2,0}^\omega\|e^{L_{\omega,T} T}.
$$
In particular if $u_{1,0}=u_{2,0}$ then for a.a. $\omega\in \Omega$ the paths $u_1(\omega,t)=u_2(\omega,t)$ for all times.
Because the local Lipschitz constant only depends on the initial data through their norms, this also proves the pathwise continuity with respect to the initial data, Definition~\ref{def:stochastic_wellposedness}(\ref{item:pathwise_continuity_initial_data}).

Finally, since by construction $z$ is adapted to $\mathcal{F}_t$ and $v$ is obtained by a fixed point argument, it follows that the stochastic process $u$ is adapted to the $\sigma(u_0,\mathcal F_t)$ and the proof is complete.
\end{proof}
%

We define
$$
\mathcal{M}(u) = M' \Big(u ;\mathcal L(u)\Big)
\qqtext{and}
\mathcal{E}(u) = E' \Big(u ;\mathcal L(u)\Big),
$$
the derivative of $E,M$ in the direction of the dissipation operator.
Given our current choice \eqref{eq:def_dissipation_L_strong_disp} for $\mathcal L$, these read explicitly
\begin{equation}
\mathcal{M} (u)= \|u \|_{\dot H^{s-\alpha}}^2 + G(\|u\|_{H^{s}}) \|u\|^2,
\label{eq:def_M}
\end{equation}
and
$$
\mathcal{E} (u)
=\|u \|_{\dot H^{s}}^2+2  \beta \left( u e^{\beta |u|^2} , (-\Delta)^{s-\alpha} u \right)
+ G(\|u \|_{H^{s}}) \left(\|u\|_{\dot H^\alpha }^2 +2 \beta\int_M  |u|^2 e^{\beta |u|^2} dx\right ).
$$
Clearly $\mathcal M$ is nicely coercive, but will fail to give sufficient $H^s$ control.
The next step is to check that $\mathcal E$ yields enough coercivity to get such a control.
To this end, fix some $\delta>0$ small enough so that $s-\delta>d/2$ and $\delta<\alpha$:
By Lemma \ref{lemsob} we have
\begin{multline*}
\left|(  u e^{\beta |u|^2}, (-\Delta)^{s-\alpha} u)\right|
\leq
\|u\|_{\dot H^{s-\alpha}} \left\| ue^{\beta |u|^2}\right\|_{\dot  H^{s-{\alpha}}}
\\
\leq
\|u\|_{\dot H^{s-\delta}} \left\| ue^{\beta |u|^2}\right\|_{H^{s-{\delta}}}
\leq C \|u\|_{\dot H^{s-\delta}} \|u\|_{H^{s-\delta}} e^{C \beta \|u\|_{H^{s-\delta}}^2}
\end{multline*}
and therefore
\begin{multline*}
\mathcal{E} (u) \geq  \|u \|_{\dot H^{s}}^2 +
G(\|u \|_{H^{s}}) \left(\|u\|_{\dot H^\alpha }^2 +2 \beta\int_M  |u|^2 e^{\beta |u|^2} dx\right)
\\
-C \|u\|_{\dot H^{s-\delta}} \|u\|_{H^{s-\delta}}e^{C \beta \|u\|_{H^{s-\delta}}^2}.
\end{multline*}
Interpolating $\|u\|_{\dot H^{s-\delta}}\leq \|u\|_{\dot H^{\alpha}}^{\theta}\|u\|_{\dot H^s}^{1-\theta}$ (for some suitable $\theta\in (0,1)$ depending on $\delta,s,\alpha$), and leveraging our growth condition \eqref{condG}, it is not too difficult to reabsorb the last exponential term into the first coercive ones, thus
\begin{equation}
 \mathcal{E} (u) \geq  c\left[\|u \|_{\dot H^{s}}^2 +
G(\|u \|_{H^{s}}) \left(\|u\|_{\dot H^\alpha }^2 +2\gamma \beta\int_M  |u|^2 e^{\beta |u|^2} dx\right)
\right]
-C
 \label{eq:lower_bound_Ecal}
\end{equation}
for some small $c>0$ and large $C>0$ only depending on the various regularity parameters.
In the same spirit, and since $G$ grows exponentially, it is easy to check that $\|u\|^2\leq C[1+ G(\|u\|)\|u\|^2]$ for some structural constant $C$ only depending on $G$,
and in particular
\begin{multline}
\|u\|^2_{H^{s-\alpha}}=\|u\|^2_{\dot H^{s-\alpha}}+\|u\|^2
\leq C\Big(1+\|u\|^2_{\dot H^{s-\alpha}}+G(\|u\|)\|u\|^2\Big)
\\
\leq C\Big(1+\|u\|^2_{\dot H^{s-\alpha}}+G(\|u\|_{H^s})\|u\|^2\Big)
\leq C\Big(1+\mathcal M(u)\Big).
\label{eq:Hs-alpha_leq_M}
\end{multline}
Hence $\mathcal E,\mathcal M$ are both coercive.
These will be the two key ingredients in controlling the compactness of solutions in some suitable sense, first by taking the inviscid limit $\sigma\to 0$, and then $N\to\infty$.
\\

Let 
$$
A^s_{N}= \sum_{|m|\leq N} \lambda_m^s |a_m|^2 
\qqtext{and}
A^s =\sum_{m\in \Z} \lambda_m^s |a_m|^2 .
$$
We assume that the Brownian coefficients $\{a_m\}_{m\in\Z}$ in \eqref{eq:def_Brownian} decay sufficiently rapidly so as to guarantee
$$
A^{\alpha}<+\infty
\qqtext{and}
A^{\frac{d-1}{2}}<+\infty,
$$
which is an intrinsic decay condition on the noise. 
The first $\alpha$-decay appears of course due to the $(-\Delta)^\alpha$ dispersion in \eqref{eq}, and the $\frac{d-1}{2}$ threshold appears for purely technical reasons related to the Sobolev embedding.

\begin{prop}
\label{consalphamass}
Let $u_0$ be a random variable in $E_N$ independent of $\mathcal{F}_t$ and such that $\E M(u_0)<\infty$.
Let $u$ be the solution to \eqref{eqsto} with initial data $u_0$.
Then, we have
$$
\E M(u(t))+ \sigma^2 \int_0^t \E\mathcal{M} (u) d\tau = \E M(u_0) +\dfrac{\sigma^2 A^0_{N}}{2}t.
$$
\end{prop}
\begin{proof}
It\^ o's formula readily gives
$$
dM(u)= M' (u;du) +\dfrac{\sigma^2}{2} \sum_{|m|\leq N} |a_m|^2 M'' (u; e_m ,e_m) dt.
$$
Since $M'(u;h)=(u,h)$ and $\left(u, -i \left[ (-\Delta)^{\alpha}  u + P_N \left(2\beta   u e^{\beta  |u|^2} \right) \right] \right)=0$, we have
\begin{equation}
M'(u;du)= -\sigma^2 \mathcal{M} (u) dt + \sigma \sum_{|m|\leq N} a_m (u,e_m) d\beta_m
\label{eq:computation_M'(u,du)}
\end{equation}
 and the result follows from $M'' (u; e_m ,e_m)  =1$ after taking expectations.
\end{proof}

\begin{prop}
\label{consalphaener}
Let $u$ be the solution to \eqref{eqsto} with initial data $u_0$, a $E_N$-valued random variable independent of $\mathcal{F}_t$.
If $\E\, E(u_0)<\infty$, then
\begin{multline*}
\E \,E(u(t))+\sigma^2 \int_0^t \E\,\mathcal{E}(u) d\tau
\\
\leq \E\, E(u_0) 
+\dfrac{\sigma^2}{2} \left(A^\alpha_{N}t +A^{\frac{d-1}{2}}_N C \int_0^t \E \int_M (1+|u|^2)e^{\beta |u|^2}  d\tau\right)
\end{multline*}
with $C=C_M(4\beta^2+2\beta)$ for some $C_M>0$ depending on the manifold $M$ only.
\end{prop}

\begin{proof}
Owing to the conservation law $E' \left(u; i  \left[(-\Delta)^{\alpha}  u + P_N \left(2 \beta   u e^{\beta |u|^2} \right)\right] \right)=0$, we have by It\^o's formula
\begin{multline*}
 dE(u)=E'(u;du) + \frac{\sigma^2}2\sum\limits_{|m|\leq N} E''(u;e_m,e_m) dt
 \\
 =-\sigma^2 \mathcal{E} (u) dt + \sigma \sum_{|m|\leq N} a_m E'(u;e_m) d\beta_m 
 +\frac{\sigma^2}2\sum\limits_{|m|\leq N} E''(u;e_m,e_m) dt
\end{multline*}
and thus
$$
\E E(u(t))+\sigma^2 \int_0^t \E\mathcal{E}(u) d\tau 
=
\E E(u_0) 
+\dfrac{\sigma^2}{2}\int_0^t \sum\limits_{|m|\leq N}|a_m|^2\E  E''(u;e_m,e_m) d\tau.
$$
By definition \eqref{eq:def_energy_E} of the energy, and owing to the universal estimate $\|e_m\|_{L^\infty}\leq C \lambda_m^{\frac{d-1}{4}}$ (for some $C$ depending on the manifold $M$ only, see \cite{donnelly}) we have
\begin{multline*}
E''(u;e_m,e_m)=\|e_m\|^2_{\dot H^\alpha}
+\int_M \left(4\beta^2|u|^2 +2\beta\right)e^{\beta|u|^2}|e_m|^2 dx
\\
\leq 
\lambda_m^\alpha
+ C_M(4\beta^2+2\beta)\lambda_m^{\frac{d-1}2}\int_M \left(1+|u|^2\right)e^{\beta|u|^2}\,dx
\end{multline*}
and the result follows by definition of $A^s_{N}$ with $s=\alpha,\frac{d-1}2$.
\end{proof}

For fixed $\sigma>0,N\in\N$, and given $w\in L^2$ as well as a Borel set $\Gamma$ of $L^2(M)$, we define the transition kernel
$$
\kappa_{\sigma ,N}^t (v;\Gamma)= \mathbb{P} \Big(u(t;P_Nv) \in \Gamma\Big),
\hspace{1cm} v\in L^2 ,\ \Gamma \in Bor (L^2),
$$
where $u(t;P_Nv)$ is the solution of the SDE \eqref{eqsto} arising from the deterministic initial datum $u(0)=P_Nv$ (constructed in Proposition~\ref{prop:stochastic_wellposedness}).
We define the semigroup $\mathfrak{B}_{\sigma ,N}^t : L^\infty (L^2 ; \R) \rightarrow L^\infty (L^2 ; \R)$ as
$$
\mathfrak{B}_{\sigma ,N}^t f(v) = \int_{L^2} f(w) \kappa_{\sigma ,N}^t (v; dw).
$$
Owing to the pathwise continuity (\ref{item:pathwise_continuity_initial_data}) in Definition~\ref{def:stochastic_wellposedness}, we see that $\mathfrak{B}_{\sigma ,N}^t$ actually satisfies the Feller property:
For any fixed time, $\mathfrak{B}_{\sigma,N}^t$ maps $C_b (L^2)$ to itself (the space of real-valued, bounded continuous functions on $L^2$).
The dual semigroup $\mathfrak{B}_{\sigma ,N}^{t\ast}: \mathfrak{p} (L^2 ) \rightarrow \mathfrak{p}(L^2)$ on probability measures $\mathfrak{p}(L^2)$ over $L^2(M)$ is accordingly defined as
$$
\mathfrak{B}_{\sigma ,N}^{t\ast} \lambda (\Gamma ) =\int_{L^2} \lambda (dw) \kappa_{\sigma ,N}^t (P_N w ; \Gamma).
$$
By definition we have, for all $\lambda\in \mathfrak p(L^2)$ and $f\in C_b(L^2)$,
\begin{equation}
\int _{L^2}f(w) d\mathfrak{B}_{\sigma ,N}^{t\ast} \lambda(dw)=\int_{L^2}\mathfrak{B}_{\sigma ,N}^{t}f(w) \lambda(dw)
=\E_\lambda f(u(t)),
\hspace{1cm} \forall\,t\geq 0
\label{eq:duality_BB*}
\end{equation}
where $u(t)$ is the unique solution of \eqref{eqsto} emanating from the initial $E_N$-valued random variable $u_0$ distributed according to $\lambda$.

\begin{thm}
\label{thm:exists_mu_s_sigma_N}
Fix $\sigma \in (0,1)$ and $s>d/2$.
For any $N\geq 2$ there exists a probability measure $\mu^s_{\sigma ,N}\in \mathfrak{p}(L^2)$ which is stationary for \eqref{eqsto} and concentrated on $H^s$.
Moreover, we have
\begin{equation}
\label{massboundalpha}
\int_{L^2} \mathcal{M} (u) \mu^s_{\sigma ,N} (du) = \dfrac{A^0_{N}}{2}\leq \dfrac{A^0}{2}
\end{equation}
and
\begin{equation}
\label{enerboundalpha}
\int_{L^2} \mathcal{E} (u) \mu^s_{\sigma ,N} (du)
\leq C\left(A^\alpha_N +A^{\frac{d-1}{2}}_N(1+A^0_N)\right)
\end{equation}
for some $C>0$ depending only on $s$, $\alpha$, $M$ but not on $\sigma$ and $N$.
\end{thm}

\begin{proof}
For $t\geq 1$ we denote
$$
\lambda_t(\bullet)=\frac 1t\int_0^t \mathfrak{B}_{\sigma ,N}^{\tau\ast} \delta_0 (\bullet) d\tau .
$$
Taking $u_0\equiv 0$ and discarding the term $\E  M(u(t))\geq 0$ in Proposition \ref{consalphamass}, we have by the characterization \eqref{eq:duality_BB*} that
$$
\E _{\lambda_t}\mathcal M(w)
=
\frac{1}{t}
\int_0^t \E \mathcal M\left(\phi^\tau_N(0)\right) d\tau
\leq \frac 1t\left(0 +\dfrac{ A^0_N}{2}t\right)=\dfrac{ A^0_N}{2}.
$$
From this and $\eqref{eq:Hs-alpha_leq_M}$ we immediately infer
\begin{equation}
\E _{\lambda_t}\big(\|w\|^2_{H^{s-\alpha}}\big)\leq C\left(1+\dfrac{A^0_N}{2}\right),
\label{eq:moment2_lambdat}
\end{equation}
thus by Chebyshev's inequality $\lambda_t(B_R^c(H^{s-\alpha}))\leq \dfrac{C}{R^2}$ uniformly in $t$.
By the Sobolev embedding $B_R(H^{s-\alpha})$ is compact in $L^2$, hence the family
$(\lambda_t)_{t\geq 1}\subset \mathfrak p(L^2)$ is tight.
By the Prokhorov theorem the weak-$\ast$ limit
$\mu^s_{\sigma ,N}= \lim\limits_{n\to\infty}\lambda_{t_n}$ exists at least for a discrete sequence of times $t_n\to+\infty$.
Since $\mathfrak{B}_{\sigma ,N}^{t\ast}$ is Feller, a standard Bogoliubov-Krylov argument (Lemma~\ref{lem:BK} in the Appendix) ensures that $\mu^s_{\sigma ,N}$ is stationary for \eqref{eqsto}.

Observe that $w\mapsto \|w\|^2_{H^{s-\alpha}}$ is lower semicontinuous for the $L^2$ topology, hence by \eqref{eq:moment2_lambdat} and standard lower semi-continuity of integral functionals (Lemma~\ref{lem:integral_functional_LSC} in the Appendix) we obtain
\begin{multline}
\E _{\mu^s_{\sigma ,N}}(\|w\|^2)
\leq
\E _{\mu^s_{\sigma ,N}}(\|w\|^2_{H^{s-\alpha}})
\\
\leq \liminf\limits_{n\to\infty} \int_{L^2}\|w\|^2_{H^{s-\alpha}} \lambda_{t_n}(dw) 
\leq C\left(1+\dfrac{A^0_N}{2}\right).
\label{eq:second_moment_mu_sigma_N_s}
\end{multline}
Taking an initial random variable distributed as $u_0\sim \mu^s_{\sigma ,N}$, we see that $\E(M(u_0))=\E (\frac 12\|u_0\|^2)=\E _{\mu^s_{\sigma ,N}}(\frac 12\|w\|^2)<+\infty$.
Thus we can legitimately apply Proposition~\ref{consalphamass} and leverage the invariance of $\mu^s_{\sigma ,N}$ to conclude that
\begin{multline*}
\E_{\mu^s_{\sigma ,N}} M(w) +\dfrac{\sigma^2 A^0_N}{2}t
=
\E M(u_0) +\dfrac{\sigma^2 A^0_N}{2}t
\\
=\E M(u(t))+ \sigma^2 \int_0^t \E\mathcal{M} (u(\tau)) d\tau 
\\
=\E M(u(0))+ \sigma^2 \int_0^t \E\mathcal{M} (u(0)) d\tau
= \E_{\mu^s_{\sigma ,N}} M(w) + \sigma^2 t\E_{\mu^s_{\sigma ,N}} \mathcal{M}(w),
\end{multline*}
and \eqref{massboundalpha} follows.

In order to estimate now $\mathcal E$, we first observe that if $\Lambda>0$ is large enough in \eqref{condG}, the Sobolev embedding gives $G(\|u\|_{H^{s}})=c\|u\|_{H^{s}}^2e^{\beta\Lambda\|u\|_{H^{s}}^2}\geq c \|u\|^2 e^{2\beta\|u\|_{L^\infty}^2}$.
We deduce from \eqref{eq:def_M} that 
$$
\mathcal M(u)\geq G(\|u\|_{H^{s}}) \|u\|^2 \geq c \left(\|u\| e^{\beta \|u\|_{L^\infty}^2}\right)^2
$$
and as a consequence
$$
\int_M (1+|u|^2)e^{\beta|u|^2}dx\leq C(1+\mathcal M(u))
$$
for some $C$ only depending on the parameters.
Since we just established \eqref{massboundalpha}, we see that
\begin{multline}
\int_{L^2}\left(\int_M (1+|u|^2)e^{\beta|u^2|}dx\right) \mu^s_{\sigma ,N}(du)
\\
\leq C\int_{L^2}\left(1+\mathcal M(u)\right) \mu^s_{\sigma ,N}(du)
\leq C\left(1+ \frac 12A^0_N\right).
\label{eq:bound_int_1+u2_exp}
\end{multline}
Because $E_N$ is finite-dimensional we also have
$$
\int_{L^2}\|u\|^2_{\dot H^\alpha} \mu^s_{\sigma ,N}(du)
\leq \int_{L^2}\|u\|^2_{ H^\alpha} \mu^s_{\sigma ,N}(du)
\leq C_{s,\alpha,N} \int_{L^2}\|u\|^2 \mu^s_{\sigma ,N}(du)<+\infty,
$$
whence by definition of the energy
$$
\E _{\mu^s_{\sigma ,N}} E(u) =\int _{L^2}\left(\|u\|^2_{\dot H^\alpha}+\int_M e^{\beta |u|^2}dx \right) \mu^s_{\sigma ,N}(du)<+\infty.
$$
This regularity allows now to take $u_0\sim \mu^s_{\sigma ,N}$ in Proposition~\ref{consalphaener}:
Omitting the details, the stationarity gives exactly as before
\begin{multline*}
\E _{\mu^s_{\sigma ,N}}\mathcal E(u)\leq \frac 12\left[
A^\alpha_N + CA^{\frac{d-1}{2}}_N \E _{\mu^s_{\sigma ,N}}\left(\int_M (1+|u|^2)e^{\beta|u|^2}dx\right)
\right]
\\
\leq C\left(A^\alpha_N +A^{\frac{d-1}{2}}_N(1+A^0_N)\right),
\end{multline*}
where the second inequality is a consequence of \eqref{eq:bound_int_1+u2_exp}.
This concludes the proof.
\end{proof}
Recalling that $\phi_N^t : E_N \rightarrow E_N$ is the global flow associated with \eqref{eqproj}, we define the corresponding flow on functions $\Phi_N^t : C_b (L^2) \rightarrow C_b (L^2)$
$$
\Phi_N^t f(v)= f(\phi_N^t (P_N v))
$$
and its dual $\Phi_N^{t \ast} : \mathfrak{p} (L^2) \rightarrow \mathfrak{p} (L^2)$
$$
\Phi_N^{t \ast} \lambda (\Gamma) = \lambda (\phi_N^{-t} (\Gamma)).
$$ 
Recalling from \eqref{eq:Hs-alpha_leq_M} that $\mathcal M(u)$ controls $\|u\|^2_{H^{s-\alpha}}$, and due to the compact  embedding $H^{s-\alpha}\subset\subset L^2$, \eqref{massboundalpha} shows that the family $\Big\{\mu_{\sigma,N}^s\Big\}_{\sigma\in(0,1)}\subset \mathfrak p(L^2)$ is tight.
By Prokhorov's theorem there exists a weak limit
$$
\lim\limits_{k\to+\infty} \mu^s_{\sigma_k ,N}=\mu_N^s\in \mathfrak p(L^2)
$$
at least for a discrete sequence $\sigma_k\to 0^+$.
For the ease of notations we will simply denote $\mu^s_{k ,N}=\mu^s_{\sigma_k ,N}$, $\mathfrak{B}_{k,N}^{t}=\mathfrak{B}_{\sigma_k,N}^{t}$, $\mathfrak{B}_{k,N}^{t\ast}=\mathfrak{B}_{\sigma_k,N}^{t\ast}$ and so on along this discrete sequence.

\begin{prop}
\label{propmainN}
The measure $\mu_N^s$ is invariant under $\phi_N^t$ and satisfies
\begin{equation}
\label{massboundN}
\int_{L^2} \mathcal{M} (u) \mu_N^s (du)= \dfrac{A^0_N}{2},
\end{equation}
\begin{equation}
\label{energyboundN}
\int_{L^2} \mathcal{E} (u) \mu_N^s (du) \leq C\left(A^\alpha_N +A^{\frac{d-1}{2}}_N(1+A^0_N)\right),
\end{equation}
 and
\begin{equation}
\int_{L^2} \mathcal{M} (u) (1- \chi_R (\|u\|^2)) \mu_N^s (du) \leq C R^{-1},
\label{eq:massboundN_trunc}
\end{equation}
where $C>0$ only depends on $\alpha,s,\beta,M$ but not on $N$.
\end{prop}

\begin{proof}
Since $\mu^s_N$ and $\mu^s_{\sigma,N}$ are supported on $E_N$, we are actually working on a finite dimensional space and the functions $\mathcal E$ and $\mathcal M$ are thus continuous.
We have moreover $\mathcal M\geq 0$, and $\mathcal E$ is lower bounded according to \eqref{eq:lower_bound_Ecal}.
As a consequence we can apply Lemma~\ref{lem:integral_functional_LSC}, and our estimates \eqref{energyboundN} and \eqref{eq:massboundN_trunc} follow from $\mu^s_{k,N}\rightharpoonup \mu^s_N$ in \eqref{enerboundalpha} and \eqref{massboundalpha}.

Equality in \eqref{massboundN} is more subtle, as $\mathcal M(u)$ is unbounded and the weak convergence is thus not enough to conclude that $\E _{\mu_{N}^s}\mathcal M(u)=\lim \E _{\mu_{k,N}^s}\mathcal M(u)$.
By \eqref{massboundalpha} we have first
$$
\dfrac{A^0_N}{2}- \int_{L^2} (1- \chi_R (\|u\|^2)) \mathcal{M} (u) \mu^s_{k,N} (du) \leq \int_{L^2} \chi_R (\|u\|^2) \mathcal{M} (u) \mu_{k,N}^s (du)\leq \dfrac{A^0_N}{2}.
$$
Using \eqref{eq:massboundN_trunc}, we get moreover
$$
\dfrac{A^0_N}{2} - C R^{-1}  \leq \int_{L^2} \chi_R (\|u\|^2) \mathcal{M} (u) \mu^s_{k,N} (du)\leq \dfrac{A^0_N}{2}.
$$
and \eqref{massboundN} simply follows by taking first $k\to+\infty$ and then $R\to +\infty$.

Next, let us show that $\mu_N^s$ is invariant under $\phi_N^t$.
Without loss of generality, we can consider $t>0$ only.
Recall that $\mu^s_{k,N}=\mathfrak{B}_{k,N}^{t\ast}\mu^s_{k,N}$ is stationary, and that $\mu^s_{k,N}$ converges weakly to $\mu_N$ in $\mathfrak p(L^2)$.
We will show below that $\mathfrak{B}_{k,N}^{t \ast}\mu^s_{k,N}$ converges weakly to $\Phi_N^{t \ast}\mu^s_N$, and the invariance of $\mu^s_N$ under $\phi^t_N$ will follow from
$$
\mu^s_N=\lim \mu^s_{k,N}=\lim \mathfrak{B}_{k,N}^{t \ast}\mu^s_{k,N}=\Phi_{N}^{t \ast}\mu^s_{N}.
$$
To this end, fix $f\in C_b (L^2)$ and write $M\geq 0$ for its uniform bound.
Since all the measures involved at this stage are actually supported on the finite-dimensional space $E_N$, we can argue with a slight abuse as if $L^2(M)$ were finite-dimensional, and it is therefore not restrictive to assume that $f$ is locally Lipschitz.
(In a locally compact metric space the space of locally Lipschitz functions is dense in $C_b$ for the uniform topology, see Lemma~\ref{lem:loc_lipschitz} in the appendix.)
We have
\begin{multline*}
\left(\mathfrak{B}_{k,N}^{t\ast } \mu^s_{k,N},f\right) - \left(\Phi_N^{t\ast } \mu_N^s ,f\right) 
= \left(\mu^s_{k,N} , \mathfrak{B}_{k,N}^t f\right) - \left(\mu_N^s , \Phi_N^t f\right)
\\
= \left(\mu^s_{k,N}, \mathfrak{B}_{k,N}^t f - \Phi_N^t f\right) - \left(\mu_N^s - \mu^s_{k,N},\Phi_N^t f\right)
= A-B.
\end{multline*}
 Using that $\Phi_N^t$ is Feller we have $\Phi_N^t f\in C_b(L^2)$, whence $B \rightarrow 0$ as $k\rightarrow \infty$. 
 Next, let us deal with $A$.
 We fix a large $R$ and write $B_R=B_R(L^2)$.
 Since $|f|$ is bounded by $M$, we have $|\mathfrak{B}_{k,N}^t f|,|\Phi_N^t f|\leq M$ as well and therefore
$$
|A|
\leq
2M\mu^s_{k,N} \left( B_R^c\right)
+\underbrace{\int_{B_R} \left|\Phi_N^t f(u_0) - \mathfrak{B}_{k,N}^t f(u_0)\right| \mu^s_{k,N} (du_0)}_{=A_1}
\leq \dfrac{M A^0_N}{R^2}+A_1,
$$
where the last inequality follows from \eqref{eq:second_moment_mu_sigma_N_s} combined with Chebyshev's inequality.
For fixed $\rho>0,t>0$ and any $u_0\in B_R(L^2)$, let
\begin{multline}
S_\rho = \Bigg\{\omega \in \Omega \qtext{s.t.}\\
\max \left(\sigma_k \sup\limits_{\tau\leq t}\left| \sum_{|m|\leq N} a_m \int_0^\tau (u_k,e_m)d\beta_m \right|,\quad \sup\limits_{\tau\leq t}\|z_k(\tau)\|  \right)\leq \rho \sigma_k t\Bigg\},
\label{eq:def_Sr}
\end{multline}
where $z_k$ is defined in \eqref{eq:def_z} for $\sigma=\sigma_k$ and $u_k=u_k(t;P_Nu_0)$ is the unique global solution of  the SDE \eqref{eqsto} with deterministic initial datum $P_Nu_0$.
Let us emphasize that $S_\rho$ depends on $u_0$, which is here a fixed deterministic initial datum in $B_R(L^2)$.
We will prove shortly that, for any $R>0,\rho>0$ and fixed $t>0$,
\begin{equation}
\label{claim}
\sup_{u_0\in B_R (L^2)} \E \big(\|\phi_N^t (P_N u_0) - u_k (t; P_N u_0)\| 1_{S_\rho}\big) \rightarrow 0
\qqtext{as} k\to \infty .
\end{equation}
Temporarily admitting this claim, let $L_R$ be the Lipschitz constant of $f$ on the ball $B_R$.
Using that $f$ is $L_R$-Lipschitz on $B_R$ and globally $M$-bounded, we have
\begin{multline*}
A_1
= \int_{B_R} \Big|\E \left[f(\phi_N^t (P_Nu_0)) -  f(u_k(t;P_N u_0))\right](1_{S_\rho}+1_{S_\rho^c})\Big| \mu^s_{k,N} (du_0)
\\
\leq L_R\int_{B_R} \E\left( \|\phi_N^t (P_N u_0) - u_k (t,P_N u_0) \| 1_{S_\rho}\right) \mu^s_{k,N} (du_0) 
+ 2M \int_{B_R} \E (1_{S_\rho^c}) \mu^s_{k,N} (du_0)
\\=A_{1,1}+A_{1,2}.
\end{multline*}
Owing to our claim \eqref{claim} we see that $A_{1,1}\rightarrow 0$ as $k\rightarrow \infty$ for fixed $R$.
On the other hand, by It\^o isometry we get for $u_0\in B_R(L^2)$
\begin{multline*}
\E \left|\sigma_k \sum_{|m|\leq N} a_m \int_0^t (u_k,e_m) d\beta_m\right|^2
= \sigma^2_k \sum_{|m|\leq N} |a_m|^2 \int_0^t \E(u_k, e_m)^2 d\tau
\\
\leq 
\sigma^2_k \sum_{|m|\leq N} |a_m|^2 \int_0^t \E\|u_k\|^2 d\tau
\leq \sigma^2_k A^0_N  t(R^2+\sigma^2_kA^0_Nt)
\end{multline*}
where the last inequality follows by Proposition~\ref{consalphamass} with $\E\|u_k(0)\|^2=\|u_0\|^2\leq R^2$.
Recalling \eqref{boundz}, we also have $\E \sup\limits_{\tau\leq t}\|z_k(\tau)\|^2 \leq C \sigma^2_kt$.
Thus, using Chebyshev's inequality, we get in $A_{1,2}$
\begin{align*}
\E (1_{S_\rho^c})
&= \mathbb{P} \Bigg\{\omega \qtext{s.t.}\\
& \hspace{1cm}\max  \left(\sigma_k \sup\limits_{\tau\leq t}\left| \sum_{|m|\leq N} a_m \int_0^\tau (u_k,e_m)d\beta_m \right|, \quad\sup\limits_{\tau\leq t}\|z_k(\tau)\|  \right)> \rho\sigma_k t  \Bigg\}
\\
& \leq \dfrac{\max\Big\{\sigma^2_k A^0_N  t(R^2+\sigma^2_kA^0_Nt),2\sigma_k^2 t\Big\}}{\rho^2 \sigma^2_kt^2}
\leq C\frac{1+R^2}{\rho^2}
\end{align*}
where $C$ is independent of $R,\rho,k$.
Summarizing, we have estimated, for any fixed $t>0, R>0,\rho>0,k\in \N,f\in C_b\cap Lip(L^2)$,
$$
|A|\leq \dfrac{C}{R^2} + A_{1,1}+  C\frac{1+R^2}{\rho^2}
$$
for some constants $C>0$ independent of $R,\rho,k$.
Taking first $k\to \infty$ with $A_{1,1}\to 0$ for fixed $R$, then $\rho\to\infty$, and finally $R\to \infty$, the result follows.

Finally, it remains to establish our claim \eqref{claim}.
To simplify notations, for fixed $u_0\in L^2$ we denote $u=u(t;u_0)=\phi_N^t( P_N u_0)$ and $v_k= v_k (t;u_0)$, where $v_k$ is the unique solution to \eqref{eqsto2} with $\sigma=\sigma_k$ and initial datum $P_N u_0$.
We set
$$
w_k=w_k(t;u_0) =u-v_k ,
$$
and recall that, by construction, the unique solution to \eqref{eqsto} is $u_k(t;u_0)=v_k(t;u_0) +z_k(t)$ (see the proof of Proposition~\ref{prop:stochastic_wellposedness}).
Writing
\begin{multline*}
\E \big(\|\phi_N^t (P_N u_0) - u_k (t; P_N u_0)\| 1_{S_\rho}\big)
=\E \big(\|u(t;u_0) - u_k (t; u_0)\| 1_{S_\rho}\big)
\\
=\E \big(\|u(t;u_0) - v_k (t; u_0)-z_k(t)\| 1_{S_\rho}\big)
\\
\leq \E \big(\|w_k (t; u_0)\|1_{S_\rho}\big) + \E(\|z_k(t)\|)
\end{multline*}
and recalling from \eqref{boundz} that $\E \|z_k(\tau)\|^2 \leq 2\sigma_k^2t\rightarrow 0$ uniformly in $\tau\in [0,t]$, our claim \eqref{claim} amounts to showing that
$$ 
\sup_{u_0 \in B_R (L^2)} \E (\|w_k(t;u_0)\| 1_{S_\rho})\xto{k\to\infty}0
\qqtext{for fixed}t>0.
$$
We will first show that $w_k(t)\to 0$ for $\mathbb P$-a.a. $\omega\in \Omega$, and then upgrade to uniform convergence of the expected value by restricting to $\omega\in S_\rho$.

Subtracting the two equations $\partial_tu -\partial_t v_k$, some elementary algebra gives
\begin{multline}
\partial_t w_k =- i\left[(-\Delta)^{{\alpha}}w_k + 2 \beta P_N \left(u e^{\beta |u|^2}-u_ke^{\beta|u_k|^2}\right) 
-2\beta P_N\left(w_ke^{\beta|u|^2}\right)\right]
\\
+ 2i \beta P_N \left(z_k e^{\beta |u_k|^2} \right)
- \sigma^2_k \Big[(-\Delta)^{s- {\alpha}} u_k + G( \|u_k\|_{H^{s}} ) u_k \Big],
\label{eq:dwk_dt}
\end{multline}
with initial datum $w_k (0)=0$.
Recall from \eqref{eq:bound_z_realization}\eqref{eq:a_priori_bound_v} that, for fixed $\omega\in\Omega$ and $t>0$, $\|v_k(\tau)\|,\|z_k(\tau)\|\leq C_{\omega,t}$ uniformly in $\tau \in [0,t]$ and $k\to\infty$, hence $\|u_k(\tau)\|=\|z_k(\tau)+v_k(\tau)\|\leq C_{\omega,t}$ as well.
Since on the finite dimensional space $E_N$ there holds $\|u\|_{L^\infty(M)}\leq C_N\|u\|_{L^2(M)}$, we see that, for a fixed $\omega$, the functions $u,u_k,v_k,z_k$ are bounded in $L^\infty((0,t)\times M)$ uniformly in $k$.
We write $\|\cdot\|_{L^\infty_{tx}}$ for this norm.
We have moreover
\begin{multline*}
 ue^{\beta |u|^2} - u_k e^{\beta |u_k |^2}
 = \sum\limits_{p\geq 0} \dfrac{\beta^{p}}{p!} (|u|^{2p} u - |v_k+z_k |^{2p} (v_k+z_k ) )
 \\
 = w_k \sum_{p\geq 0} f_{2p}(u, v_k) + z_k \sum_{p\geq 0} \tilde f_{2p}(v_k,z_k)
 = w_k F(u,v_k) + z_k\tilde F(v_k,z_k).
 \end{multline*}
where $f_{2p}$ and $\tilde f_{2p}$ are explicit polynomials of degree $2p$.
It is not difficult to check that the series $F(u,v)=\sum_{p\geq 0} f_{2p}(u, v)$ and $\tilde F(v,z)= \sum_{p\geq 0} \tilde f_{2p}(v,z) $ are converging.
By \eqref{eq:dwk_dt} we find in any fixed time-interval $\tau\in[0,t]$
\begin{multline}
 \frac{\partial}{\partial\tau}\left(\frac 12\|w_k\|^2\right)
 \leq 0
 + 2\beta \left(\|w_k\|^2\|F(u,v_k)\|_{L^\infty_{t,x}} + \|w_k\|\|z_k\|\|\tilde F(v_k,z_k)\|_{L^\infty_{t,x}}\right)
 \\
 +2\beta \|w_k\|\|z_k\|e^{\beta\|u_k\|^2_{L^\infty_{t,x}}}
 +\sigma_k^2 \|w_k\|\left(\|u_k\|_{H^{\frac{s-\alpha}{2}}}+G(\|u_k\|_{H^{s}})\|u_k\|\right).
 \label{eq:ODE_wk}
\end{multline}
Observe that, by definition \eqref{eq:def_Sr} of $S_\rho$, we have $\sup\limits_{\tau\in[0,t]}\|z_k(\tau)\|\leq \rho\sigma_k t$ for $\omega\in S_\rho$.
As a consequence, for $\omega\in S_\rho$ the upper bound $Z_{\omega,t}$ in \eqref{eq:bound_z_realization} can be taken of the form
$$
\sup\limits_{\tau\leq t}\|z_k(\tau)\|\leq \rho\sigma_k t
$$
and thus from \eqref{eq:a_priori_bound_v} with a deterministic initial datum $v_k(0)=P_Nu_0\in B_R(L^2)$ we control
$$
\sup\limits_{\tau\leq t}\|v_k(\tau)\|^2\leq R^2 + (c+\rho^2\sigma_k^2 t^2)+16\beta^2\rho^2t^3 \leq C_{R,\rho,t}
$$
uniformly in $\omega\in S_\rho$ and $u_0\in B_R$.
With these uniform bounds on $z_k,v_k$ and $u_k=v_k+z_k$, an easy application of Young's inequality in \eqref{eq:ODE_wk} immediately gives
$$
 \frac{\partial}{\partial\tau}\left(\|w_k\|^2\right)
 \leq C_1 \|w_k\|^2 + C_2\sigma_k^2,
$$
where $C_1,C_2$ depend on $\rho,R,t$ but not on $k,\omega\in S_\rho, u_0\in B_R$.
Gr\"onwall's inequality with $w_k (0)=0$ easily gives
$$
\sup\limits_{u_0\in B_R}\|w_k(t)\|^21_{S_\rho}\leq C_2te^{C_1t}\sigma_k^2 \xto{k\to\infty} 0
$$
uniformly in $\omega\in \Omega$ and as a consequence
$$
\E\left(\sup\limits_{u_0\in B_R}\|w_k(t;u_0)\|1_{S_\rho}\right)\to 0.
$$
For a given $u_0\in B_R$ we have $\|w_k (t;u_0)\| 1_{S_\rho} \leq \sup\limits_{\tilde u_0 \in B_R} \|w_k (t; \tilde u_0 )\| 1_{S_\rho}$, hence
$$
\sup\limits_{u_0\in B_R}\E\left( \|w_k(t;u_0)\|1_{S_\rho}\right)
\leq \E\left(\sup\limits_{\tilde u_0\in B_R} \|w_k(t;\tilde u_0)\|1_{S_\rho}\right)
\to 0
$$
and the proof is complete
 \end{proof}
\subsection{The large dimension limit $N\to\infty$}

\begin{prop}
\label{prop:N_to_infty}
There exist a measure $\mu^s\in \mathfrak p(L^2)$ such that, up to some subsequence (not relabeled here),
$$
\lim_{N\to \infty} \mu^s_{N}=\mu^s,
\qquad\mbox{weakly over }H^r \mbox{ for all }  r<s.
$$
Moreover, we have
\begin{equation}
\int_{L^2} \mathcal{M} (u) \mu^s (du)=\dfrac{A^0}2,
\label{eq:int_Mmus_=A0}
\end{equation}
and
\begin{multline}
\int_{L^2} \left[\|u\|_{\dot H^s}^2 + \left(\|u\|_{\dot H^\alpha }^2  + \int_M u^2 e^{\beta |u|^2}\right) G(\|u\|_{H^{s}})  \right] \mu^s (du)
\\
 \leq  C\left(A^{\alpha} +A^{\frac{d-1}{2}}(1+A^{0})\right).
\label{eq:estimate_exponential_Hs_mu}
\end{multline}
for some $C$ only depending on the various parameters.
\end{prop}
Note that, for finite $N$, the measure $\mu^s_N\in\mathfrak p(L^2)$ is supported on the finite-dimensional space $E_N$, and can therefore be viewed unambiguously as a probability measure over $H^r$ for any $r\geq 0$.
For definiteness and consistency we consider the limit $\mu^s$ as a probability measure over $L^2$ as well, but it is part of our statement that $\mu^s$ can also be considered as a probability measure on $H^r$ for any $r<s$.

\begin{proof}
From \eqref{energyboundN} and \eqref{eq:lower_bound_Ecal} we see that
$$
c\int_{L^2}\|u\|^2_{\dot H^s}\mu^s_N(du)\leq \int_{L^2}\mathcal E(u)\mu^s_N(du)\leq C,
$$
and from \eqref{massboundN} combined with \eqref{eq:Hs-alpha_leq_M}
\begin{equation*}
\int_{L^2}\|u\|^2\mu^s_N(du)
\leq \int_{L^2}\|u\|^2_{H^{s-\alpha}}\mu^s_N(du)
\leq \int_{L^2}\mathcal M(u)\mu^s_N(du)\leq C.
\end{equation*}
Whence
\begin{equation}
 \int_{L^2}\|u\|^2_{H^{s}}\mu^s_N(du)\leq C
\label{eq:control_u_Hs_by_E}
 \end{equation}
uniformly in $N$.
By Chebyshev's inequality we have thus $\mu^s_N(B_R^c(H^s))\leq \frac{C}{R^2}$, and since $B_R(H^s)$ is compactly embedded in $L^2$ we see that the sequence  $\{\mu^s_N\}_{N\geq 1}$ is tight over $L^2$.
Prokhorov's Theorem guarantees that, up to extraction of a subsequence if needed, $\mu^s_{N}\rightharpoonup \mu^s$ for some weak limit $\mu^s\in \mathfrak p(L^2)$.
Actually, for any $r<s$ the embedding $B_R(H^s)$ into $H^r$ is compact, hence the very same argument shows that $\mu^s_{N}$ is weakly compact over $H^r$ for any $r<s$.
By uniqueness of the limit we see that the weak convergence holds for any such $r$ along this particular subsequence.
For the remaining estimates one can proceed exactly as in the proof of Proposition \ref{propmainN}, with the minor difference that the various functions $\mathcal M(u),\|u\|^2\dots$ appearing in the integrals are not continuous anymore but only lower semicontinuous (with respect to the $L^2$ topology, which is enough in order to safely apply Lemma~\ref{lem:integral_functional_LSC}).
\end{proof}

Now that we have a natural candidate $\mu^s=\lim \mu^s_N$ as a ``good'' probability measure for \eqref{eq}, it remains to construct global solutions and to show that $\mu^s$ is indeed invariant for the corresponding flow.
Here we carry out most of the preparatory technical work.

A key issue will be to prove that the $H^r$ norms grow slowly enough (logarithmically in time) for $r<s$.
In order to measure this growth in the correct time scale, let us introduce the monotone increasing function
$$
\tilde{G}(\rho)=\ln (1+G(\rho)) \qqtext{for}\rho\geq 0.
$$
For the reader's convenience it is worth anticipating that the growth function $\zeta$ in Theorem~\ref{thmstrong} will be $\zeta=2\tilde G^{-1}$, roughly speaking.

\begin{prop}
\label{prop:small_measure_growth}
Let $d/2< r<s$, $N\geq 0$.
For any $i\in \N$ there exists a set $\Sigma_{N,r}^i\subset E_N$ such that
\begin{equation}
 \label{eqsum}
 \mu_N^s (E_N \setminus \Sigma_{N,r}^i)\leq C e^{-i},
\end{equation}
for some $C>0$ depending only $s$.
Moreover for all $u_0 \in \Sigma_{N,r}^i$, we have
\begin{equation}
\|\phi_N^t (u_0)\|_{H^r} \leq 2 \tilde{G}^{-1} (1+i +\ln (1+|t|))
\hspace{1cm}
\forall\,t\in\R.
\label{eq:estimate_Hr_SigmaNri}
\end{equation}
\end{prop}
One can think here of $i$ as a fixed time delay.
Given such a fixed delay, \eqref{eq:estimate_Hr_SigmaNri} means
that there is a large set on which we have an explicit logarithmic $H^r$ control, and this set is exponentially large for large delays.
\begin{proof}
For $j\geq 1$, we set
$$
B_{N,r}^{i,j} = \Big\{u\in E_N \qtext{s.t.} \|u\|_{H^r} \leq 2\tilde{G}^{-1} (i+j)\Big\},
$$
a ball of radius $i+j$ in the $\tilde G$ scale.
It may help to think of $i$ as a fixed time delay appearing in \eqref{eq:estimate_Hr_SigmaNri}, while $j\in \N$ can be thought of as a discrete ``current running time''.

Let $T=T(s,R)$ be the local time of existence from Proposition~\ref{prop:exbigsN} with $R=\tilde{G}^{-1} (i+j)$, and
\begin{equation}
\Sigma_{N,r}^{i,j} = \bigcap\limits_{k=0}^{k_j} \phi_N^{-kT} (B_{N,r}^{i,j})
\qqtext{with}
k_j=\flr{e^j/T}.
\label{eq:def_Sigmaij}
\end{equation}
On can think of this as the set of initial data allowing to maintain a logarithmic control up to large times $|t|\lesssim e^j$ and modulo a fixed delay $i$.
Using the invariance of $\mu_N^s$, we get
\begin{multline*}
 \mu_N^s \left(E_N \setminus \Sigma_{N,r}^{i,j}\right)
 =\mu_N^s\left(\bigcup\limits_{k\leq k_j} E_N \setminus \phi_N^{-kT} (B_{N,r}^{i,j})\right)
 \\
 \leq \sum\limits_{k\leq k_j}
 \mu_N^s\left(E_N \setminus \phi_N^{-kT} (B_{N,r}^{i,j})\right)
 =\sum\limits_{k\leq k_j}\mu_N^s\left(\phi_N^{-kT}\left(E_N \setminus  B_{N,r}^{i,j}\right)\right)
 \\
 =\sum\limits_{k\leq k_j}\mu_N^s\left(E_N \setminus  B_{N,r}^{i,j}\right)
 \leq \left(k_j +1\right) \mu_N^s \left(E_N \setminus B_{N,r}^{i,j}\right).
 \end{multline*}
 We claim now that
\begin{equation}
G(\|u\|_{H^r})\leq C(\mathcal E (u)+1)
\label{eq:claim_G_exp}
\end{equation}
for some $C$ depending only on $r,s$.
Indeed, we first observe from $\int_M |u|^2e^{\beta |u|^2}\geq \|u\|^2$ in \eqref{eq:lower_bound_Ecal} that
$$
\mathcal E(u)\geq c G(\|u\|_{H^s})\|u\|^2_{H^\alpha}-C.
$$
Note that, in dimension $d\geq 2$ and because we fixed $\alpha\leq 1$ with $d/2<r<s$, there holds $\alpha <r<s$.
By standard interpolation there exists $\theta \in (0,1)$ depending on $r,s$, and $\alpha$ such that $\|u\|_{H^r}\leq \|u\|^\theta_{H^\alpha} \|u\|_{H^s}^{1-\theta}$.
Moreover, owing to the exponential growth of $G$, there exists $R$ depending only on $G,\theta$ such that the function $\rho\mapsto G(\rho)/\rho^{2\frac{1-\theta}\theta}$ is monotone increasing at least for large $\rho\geq R\geq 1$.
Thus if $\|u\|_{H^r} \geq R\geq 1$ we have
\begin{multline*}
G(\|u\|_{H^r})
\leq G(\|u\|_{H^r}) \|u\|_{H^r}^2 
=\frac{G(\|u\|_{H^r})}{\|u\|_{H^r}^{2\frac{1-\theta}\theta}} \|u\|_{H^r}^{\frac 2\theta}
\leq \frac{G(\|u\|_{H^s})}{\|u\|_{H^s}^{2\frac{1-\theta}\theta}} \|u\|_{H^r}^{\frac 2\theta} 
\\
=  G(\|u\|_{H^s})\left(\dfrac{\|u\|_{H^r}^{1/\theta  }}{\|u\|_{H^s}^{\frac{1-\theta}{\theta}}}\right)^2
\leq G(\|u\|_{H^s}) \|u\|_{H^\alpha}^2 \leq \frac 1c (\mathcal E(u)+C).
\end{multline*}
On the other hand, since $\mathcal E(u)\geq -C$ from \eqref{eq:lower_bound_Ecal}, for $\|u\|_{H^r}\leq R$ we have trivially have $G(\|u\|_{H^r})\leq G(R)=C\leq C(1+\mathcal E(u))$ and our claim \eqref{eq:claim_G_exp} follows.

Back to our main argument: From \eqref{energyboundN} and \eqref{eq:claim_G_exp} we have therefore
$$
\E_{\mu^N_s} \exp\left(\tilde G(\|u\|_{H^r})\right) \leq C
$$
uniformly in $N$, and Chebyshev's inequality yields
\begin{align*}
\mu_N^s (E_N \setminus \Sigma_{N,r}^{i,j}) 
&
\leq \left(k_j +1\right) \dfrac{1}{e^{2(i+j)}   }\int\limits_{\tilde G(\|u\|_{H^r})\geq 2(i+j)} \exp(\tilde{G} (\|u\|_{H^r})) \mu^s_N (du)
\\
&\leq 2e^{j} /T \dfrac{1}{e^{2(i+j)}}C = Ce^{-i}\dfrac{e^{-(i+j)}}{T}
    \end{align*}
    (recall that we defined $k_j=\flr{e^j/T}$).
By \eqref{eq:def_existence_time_T} we have $T=T(s,R)=\frac{c(s)}{\beta e^{C(s)\beta R^2}}$ with here $R=2\tilde G^{-1}(i+j)$, hence
$$
\mu_N^s (E_N \setminus \Sigma_{N,r}^{i,j}) 
\leq C e^{-i}\frac{e^{-(i+j)}}{\beta e^{-C\beta[2\tilde G^{-1}(i+j)]^2}}
\leq C e^{-i}e^{-(i+j)/2}.
$$
The last inequality follows from our growth condition \eqref{condG}, which gives explicitly here $\tilde G^{-1}(z)\sim \sqrt{z/\Lambda\beta}$ for large $z$ and therefore $e^{-C\beta[2\tilde G^{-1}(i+j)]^2}\sim e^{-\frac{4C}{\Lambda}(i+j)}\geq e^{-(i+j)/2}$ if $\Lambda\geq 8C(s)$ is chosen large enough.

Set now
\begin{equation}
\Sigma_{N,r}^i =\bigcap\limits_{j\geq 1} \Sigma_{N,r}^{i,j}.
\label{eq:def_Sigma_Nri}
\end{equation}
In words, $\Sigma_{N,r}^i$ is the set of initial data that maintain logarithmic control for \emph{all} times (here for $|t|\leq e^j$ for \emph{all} $j\in\N$), up to some fixed delay $i$ and measured in the $H^r$ scale.
Our claim \eqref{eqsum} finally follows from the summability
\begin{equation*}
\mu_N^s (E_N \setminus \Sigma_{N,r}^i)
\leq \sum\limits_{j\geq 0}\mu_N^s (E_N \setminus \Sigma_{N,r}^{i,j})
\leq C e^{-i}\sum\limits_{j\geq 0}e^{-(i+j)/2}\leq Ce^{-i}. 
\end{equation*}

It remains now to establish \eqref{eq:estimate_Hr_SigmaNri}.
Fix first $i,j\in \N$.
We claim that for any $u_0 \in \Sigma_{N,r}^{i,j}$ there holds
\begin{equation}
\|\phi_N^t(u_0)\|_{H^r} \leq 2\tilde{G}^{-1} (i+j),
\hspace{1cm}\forall\,t\leq e^j.
\label{eq:Hr_i+j}
\end{equation}
Indeed for $t\leq e^j$, we can write uniquely $t=kT +\tau$ where $k\leq k_j=\flr{e^j/T}$ is an integer and $\tau \in [0,T)$.
As a consequence $\phi^t_N (u_0)= \phi^{\tau}_N( \phi^{kT}_N (u_0)) = \phi^\tau_N (w)$, with $w=\phi^{kT}_N (u_0)\in B_{N,r}^{i,j}=B_R(H^r)$ by definition \eqref{eq:def_Sigmaij} of $\Sigma_{N,r}^{i,j}$ and with $R=\tilde G^{-1}(i+j)$.
Since $\tau\leq T$, Proposition~\ref{prop:exbigsN} yields $\|\phi^\tau_{N}(w)\|_{H^r}\leq 2R$ as claimed.

Fix now $i\in \N$ and $u_0\in \Sigma_{r,N}^i$, and choose an arbitrary $t\geq 0$.
There exists a unique $j\geq 1$ such that $e^{j-1}\leq 1+t< e^j$, in particular $j\leq 1+\ln(1+t)$.
By definition of $\Sigma_{r,N}^i$ we have $u_0\in \Sigma_{r,N}^{i,j}$ for this $j$ in particular, hence from our previous claim
$$
\|\phi_N^t (u_0)\|_{H^r} \leq 2 \tilde{G}^{-1} (i+j) \leq 2 \tilde{G}^{-1} (i+1 +\ln (1+t))
$$
and the proof is complete.
\end{proof}

The next result allows to trade off some Sobolev regularity for control over the time delay.

\begin{lem}
Fix any $d/2<r<s$, $0<r_1 <r$, and $t\in \R$.
There exists $i_1=i_1(t) \in \N^\ast$ such that, if $u_0 \in \Sigma_{N,r}^i$ for some $i$, then $\phi_N^t (u_0) \in \Sigma_{N,r_1}^{i+i_1}$.
\label{lem:Sigma_i+i1}
\end{lem}

\begin{proof}
Assume without loss of generality that $t>0$, and let $u_0 \in \Sigma_{N,r}^i$ for some $i$.
Then, as already established in \eqref{eq:Hr_i+j}, we have for any $j\geq 1$
$$
\|\phi_N^{t_1} (u_0)\|_{H^r} \leq 2  \tilde{G}^{-1} (i+j),
\hspace{1cm} \forall\,t_1 \leq e^j. 
$$
Let $i_1=i_1 (t)$ be such that $e^j +t \leq e^{j+i_1}$ for all $j\geq 1$.
Again, we get
$$
\|\phi_N^{t_1 +t} (u_0)\|_{H^r} \leq 2 \tilde G^{-1}(i+j+i_1),
\hspace{1cm} \forall\,t_1 \leq e^j.
$$
For every $u_0 \in \Sigma_{N,r}^i$ we have by definition $\|u_0\| \leq \|u_0\|_{H^r} \leq 2  \tilde{G}^{-1} (1+i)$, and since mass is preserved
$$
\|\phi_N^{t+t_1} (u_0)\|=\|u_0\|\leq \|u_0\|_{H^r}
\leq 2  \tilde{G}^{-1} (1+i).
$$
By interpolation, for any $r_1 \in (0, r)$ there is $\theta \in (0,1)$ depending only on $r,r_1$ such that
\begin{align*}
\|\phi_N^{t+t_1} (u_0)\|_{H^{r_1}}
& \leq \|\phi_N^{t+t_1} (u_0) \|^{1-\theta} \|\phi_N^{t+t_1} (u_0)\|_{H^r}^\theta 
\\
& \leq  \left( 2\tilde{G}^{-1} (1+i)\right)^{1-\theta}   \left(2\tilde{G}^{-1} (i+j+i_1)\right)^\theta \leq  2\tilde{G}^{-1} (i+j+i_1)
\end{align*}
for all $j$ and $t_1\leq e^j$.
This is exactly the definition of $\phi_N^{t+t_1} (u_0)\in \Sigma^{i+i_1}_{N,r_1}$, and the proof is complete.
\end{proof}

We set
\begin{equation}
 \Sigma_r^i =\left\{u\in H^r  \mbox{ s.t. there exists } u_{N_k}\to u\mbox{ with } u_{N_k}\in \Sigma_{N_k ,r}^i\mbox{ as } N_k\rightarrow \infty\right\},
\label{eq:def Sigma_ri}
 \end{equation}
and
\begin{equation}
\Sigma_r = \bigcup\limits_{i\geq 1} \overline{\Sigma_r^i}.
\label{eq:def Sigma_r}
\end{equation}
In words, $\Sigma_r$ is the set of initial data that maintain $H^r$ logarithmic control for all times, up to \emph{some} delay $i$.
\begin{prop}
Let $\mu^s=\lim \mu^s_{N}$ be the probability measure constructed in Proposition~\ref{prop:N_to_infty}.
Then
$$
\mu^s (\Sigma_r )=1
\qqtext{for all} r<s.
$$
\label{prop:compare_mus_with_musir}
\end{prop}

\begin{proof}
Using the Portmanteau theorem, the definition \eqref{eq:def_Sigma_Nri} of $\Sigma_{N,r}^i$ and the fact that $\Sigma_{N,r}^i \subset \Sigma_r^i$, we have
$$
\mu^s \left(\overline{\Sigma_r^i}\right) 
\geq \lim_{N\rightarrow \infty} \mu_{N}^s \left(\overline{\Sigma_r^i}\right)
\geq \lim_{N\rightarrow \infty}\mu_{N}^s \left(\Sigma_r^i\right)
\geq \lim_{N\rightarrow \infty} \mu_{N}^s \left(\Sigma_{{N},r}^i\right)
\geq 1 - C  e^{-i},
$$
where we crucially used \eqref{eqsum} for the last inequality.
Using now that $\overline{\Sigma_r^i}$ is non-decreasing in $i$, we get
$$
1\geq \mu^s (\Sigma_r)= \mu^s \Big(\bigcup_{i\geq 1} \overline{\Sigma_r^i }\Big) = \lim_{i\rightarrow \infty} \mu \left(\overline{\Sigma_r^i}\right) \geq 1
$$
and $\Sigma_r$ has thus full $\mu^s$ measure as claimed.

\end{proof}

We are now ready for the core of the matter, namely the construction of global solutions.
\subsection{Global well-posedness and variations}
\label{sec:variations}
\begin{prop}
\label{prop:exists_global_Sigma_r}
Let $d/2<r <s$.
For $\mu^s$-almost all $u_0\in\Sigma_r$ there is a unique global in time solution $\phi^t(u_0)$ to \eqref{eq}.
This solution satisfies moreover
\begin{equation}
\|\phi^t (u_0)\|_{H^{r}}\leq 2 \tilde{G}^{-1} (1+i+\ln (1+|t|)),
\hspace{1cm}\forall\,t\in\R
\label{eq:phit_uo_bound}
\end{equation}
for some $i=i(u_0,r)$.
\end{prop}

\begin{proof}
Recalling that $\mathcal M(u)$ controls $\|u\|^2$, we see from \eqref{eq:int_Mmus_=A0} and \eqref{eq:estimate_exponential_Hs_mu} that
$$
\int_{L^2} \|u\|^2_{H^s}\mu^s(du)<\infty
$$
and thus $\mu^s$-almost all initial data $u_0\in \Sigma _r\subset L^2$ actually lie in $H^s$.
By definition \eqref{eq:def Sigma_r} of $\Sigma_r$ we have that $u_0 \in \overline{\Sigma_r^i}$ for some $i\in \N$.
Assume first that $u_0\in \Sigma_r^i$.
By definition \eqref{eq:def Sigma_ri} of $\Sigma_r^i$, there exists a sequence $u_{0,N} \in \Sigma_{N,r}^i$ converging to $u_0$ in $H^r$ as $N\to\infty$.
Our previous estimate \eqref{eq:estimate_Hr_SigmaNri} readily gives
\begin{equation}
\|\phi_N^t (u_{0,N})\|_{H^r} \leq 2 \tilde{G}^{-1} (1+i+\ln (1+|t|)),
\hspace{1cm}\forall\,t\in\R.
\label{eq:phitN_uoN_bound}
\end{equation}
In particular for $t=0$ we have $\|u_{0,N}\|_{H^r} \leq 2  \tilde{G}^{-1} (1+i)$, whence taking the limit in $N$
$$
\|u_0\|_{H^r} \leq 2  \tilde{G}^{-1} (1+i).
$$
The same holds if $u_0 \in \partial \Sigma_r^i$, simply taking a sequence $u_{0,k}\to u_0$ in $H^r$ for some sequence $u_{0,k}\in \Sigma_r^i$ satisfying the same estimate.

Fix an arbitrary $T_0>0$.
Clearly it suffices to prove that $\phi^t(u_0)$ is well-defined and satisfies \eqref{eq:phit_uo_bound} at least for times $|t|\leq T_0$.
To see this, set $R=2 \tilde{G}^{-1} (1+i +\ln (1+T_0)) +1\geq \|u_0\|_{H^r}+1$, and let $T=T(r,2R)$ be the common local existence time from Proposition \ref{prop:exbigsN} and Proposition~\ref{prop:exbigs}.
If $T\geq T_0$ then Lemma~\ref{lem:phitN_converges_phit} allows passing to the limit in \eqref{eq:phitN_uoN_bound} for all $|t|\leq T_0$ and our claim follows.
Assume otherwise that $T<T_0$.
Consider first $u_0 \in \Sigma_r^i$, and let $u_{0,N} \in \Sigma_{N,r}^i$ be the corresponding sequence converging to $u_0$ in $H^r$.
By Lemma~\ref{lem:phitN_converges_phit} we have thus
$$
\sup\limits_{t\in [-T,T]}\|\phi^t (u_0) - \phi^t_N (u_{0,N})\|_{H^r} \to 0
\qquad\mbox{as } N\to\infty .
$$
Observe from \eqref{eq:phitN_uoN_bound} that, for $|t|\leq T< T_0$ we have $\|\phi^t_N(u_{0,N})\|\leq 2\tilde G^{-1}(1+i+\ln(1+|T|))\leq R-1$.
The triangle inequality thus gives for $|t|\leq T$ and large $N$
\begin{multline*}
\|\phi^t (u_0)\|_{H^r} 
\leq \|\phi^t (u_0) - \phi^t_N (u_{0,N})\|_{H^r} + \|\phi_N^t (u_{0,N})\|_{H^r} 
\\
\leq \underbrace{\|\phi^t (u_0) - \phi_N^t (u_{0 ,N}) \|_{H^r}}_{\to 0} + R-1 \leq R,
\end{multline*}
and as a consequence $\phi^{T}(u_0)\in B_R(H^r)$.
Bootstrapping and exhausting $t\in [-T_0,T_0]$ by finitely many intervals of fixed size $T$, we can repeat the argument as many times as needed and the claim follows.

Finally, if $u_0 \in \partial \Sigma_r^i$, let $u_{0,k}\in\Sigma_r^i$ be a sequence converging to $u_0$ in $H^r$.
For the same $R,T$ depending only on the arbitrary choice of $T_0$ and the Sobolev exponent $r$, we have $\|u_0\|_{H^r}\leq R-1$ hence clearly $u_0,u_{0,k}\in B_R(H^r)$ at least for large $k$.
By Lemma~\ref{lem:phitN_converges_phit} we see that for $|t|\leq T=T(r,2R)$ there holds
$$
\lim\limits_{k\rightarrow \infty} \left\|\phi^t (u_0) - \phi^t (u_{0,k})\right\|_{H^r} =0,
$$
whence by triangular inequality and for all times $|t|\leq T$
\begin{multline*}
\|\phi^t (u_0)\|_{H^r} 
\leq \|\phi^t (u_0) - \phi^t (u_{0,k})\|_{H^r} + \|\phi^t (u_{0,k})\|_{H^r} 
\\
\leq \underbrace{\|\phi^t (u_0) - \phi^t (u_{0,k}) \|_{H^r}}_{\to 0} + R-1 \leq R.
\end{multline*}
The result then follows by the same bootstrap argument as before and we omit the details.
\end{proof}

Let $(s_n)_n$ be an increasing sequence such that $\lim\limits_{n\rightarrow \infty} s_n =s$, and set
$$
\Sigma^s = \bigcap_{n\in\ N} \Sigma_{s_n} .
$$

\begin{prop}
$\Sigma^s$ has full $\mu^s$-measure.
Moreover, up to removing a $\mu^s$-null set from $\Sigma^s$ if needed, for any $u_0\in \Sigma^s$ there is a unique global solution to \eqref{eq} and the corresponding global flow satisfies $\phi^t (\Sigma^s )=\Sigma^s$ for all $t\in \R$.
\label{prop:Sigma_full_measure_invariant}
\end{prop}

\begin{proof}
Up to removing a null set from each $\Sigma_{s_n}$ if needed, we can assume from Proposition~\ref{prop:exists_global_Sigma_r} that a global solution $\phi^t (u_0)$ exists for \emph{any} $u_0\in \Sigma_{s_n}$, and thus for any $u_0\in \Sigma^s$.
Since $\Sigma_r$ has full measure for any $r<s$ (Proposition~\ref{prop:compare_mus_with_musir}) and since the intersection is countable, clearly $\Sigma^s$ has full measure.

Let $u_0 \in \Sigma^s$, and fix an arbitrary rank $n\in\N$.
We will show that $\phi^t(u_0)\in \Sigma_{s_m}$ for all $m\leq n-1$, hence $\phi^t(u_0)\in \Sigma^s$.
The key point will be the cascade of regularity from Lemma~\ref{lem:Sigma_i+i1}, allowing to propagate $H^{s_n}$ estimates to $H^{r}$ control for all $r<s_n$.
More precisely: 
By definition of $\Sigma_{s_n}$ there exists $i_n\geq 1$ such that $u_0 \in \overline{\Sigma_{s_n}^{i_n}}$.
\begin{itemize}
 \item
Assume first that $u_0\in \Sigma_{s_n}^{i_n}$:
by definition of $\Sigma_{s_n}^{i_n}$ there exists a sequence $\{u_{0,N}\}_{N}\subset \Sigma_{N,s_n}^{i_n}$ converging to $u_0$ in $H^{s_n}$ as $N\to\infty$.
According to Lemma~\ref{lem:Sigma_i+i1} there exists $I_n=I_n(t)\in\N$ such that $\phi_N^t (u_{0,N}) \in \Sigma_{N, r}^{i_n+I_n}$ for all $r<s_n$, in particular for $r=s_1,\dots,s_{n-1}$.
Leveraging Lemma~\ref{lem:phitN_converges_phit} and passing to the limit $\phi^t_N (u_{0,N})\to \phi^t (u_0)$, we see that $\phi^t (u_0)$ lies in $\Sigma_{r}^{i_n+I_n}$ for $r=s_1,\dots,s_{n-1}$.
\item
If now $u_0 \in \partial \Sigma_{s_n}^i$, there is a sequence $\{u_{0,k}\}_k \subset \Sigma_{s_n}^{i_n}$ converging to $u_0$ in $H^{s_n}$.
We actually just proved that $\phi^t (\Sigma_{s_n}^{i_n}) \subset \Sigma_{r}^{i_n +I_n}$ for all $r<s_n$, hence $\phi^t(u_{0,k})\in \Sigma_{r}^{i_n +I_n}$ for such $r$.
Since $\phi^t (.)$ is continuous (Lemma~\ref{lem:phitN_converges_phit}), we have $\phi^t (u_0) =\lim\limits_{k\to\infty} \phi^t (u_{0,k}) \in \overline{\Sigma_{r}^{i_n+I_n}}$ for all $r<s_n$, in particular for $r=s_1,\dots,s_{n-1}$.
\end{itemize}
In any case we see that $\phi^t (u_0)\in\cap_{m\leq n-1}\Sigma_{s_m}$.
Because $n\in \N$ was arbitrary we conclude that $\phi^t(u_0)\in \left(\cap_{n\geq 0}\Sigma_{s_n}\right)=\Sigma^s$, hence $\phi^t(\Sigma^s) \subset \Sigma^s$.

Conversely, let $u\in \Sigma^s$.
Letting $u_0 = \phi^{-t} (u)$, we have from the previous step that $\phi^{-t}(\Sigma^s)\subset\Sigma$ and thus $u_0\in \Sigma^s$.
Thus $u=\phi^t (u_0)\in\phi^t(\Sigma^s)$ and therefore $\Sigma^s\subset\phi^t(\Sigma^s)$ as well.
\end{proof}

\begin{prop}
The measure $\mu^s\in \mathfrak p(L^2)$ is invariant under $\phi^t$ on $\Sigma^s$, in the sense that $\mu(\phi^{-t}(A))=\mu(A)$ for any $r<s$ and all $H^r$-measurable set $A\subset\Sigma^s$.
\end{prop}

\begin{proof}
Note that, since $\Sigma^s\subset H^{s}$ has full measure and $\phi^t(u)$ is globally defined for all $u\in\Sigma^s$, the flow $\phi^t(u)$ is unambiguously defined for $\mu^s$-a.e. $u\in L^2$ and therefore it suffices to prove that $\int_{H^r}f(u)\mu^s(du)=\int_{H^r} f(\phi^t(u))\mu^s(du)$ for all $f\in C_b(H^r)$.
(Here we choose unambiguously to view $\mu^s$ as a probability measure over $H^r$, for any fixed $r<s$.)
In fact, by Ulam's tightness theorem, it suffices to prove that
$$
\int_{K}f(u)\mu^s(du)=\int_{K} f(\phi^t(u))\mu^s(du)
$$
for all $f\in C_b(H^r)$ and any $H^r$-compact set $K$.
As in the proof of Proposition~\ref{propmainN}, and because we already know that $\mu^s_N$ is invariant under $\phi^t_N$ (Proposition~\ref{propmainN}) and converges weakly to $\mu^s$ (Proposition~\ref{prop:N_to_infty}), we have
$$
\int_K f(\phi^t_N(u))\mu^s_N(du) =
\int_K f(u)\mu^s_N(du)\to \int_K f(u)\mu^s(du)
$$
and the statement will follow if we can prove that
$$
\int_K f(\phi^t_N(u))\mu^s_N(du) \xto{N\to\infty}\int_K f(\phi^t(u))\mu^s(du)
$$
(up to subsequences, not relabeled here for the ease of notations).

Clearly, by repeated iterations if needed, it is enough tho establish the statement for $|t|\leq T$ small enough.
Fix a large $R>0$ so that $K\subset B_B(H^r)$, let $T=T(r,R)$ as usual, and pick any $|t|\leq T$.
We write then
\begin{multline*}
 \left|\int_K f(\phi^t_N(u))\mu^s_N(du) -\int_K f(\phi^t(u))\mu^s(du)\right|
 \\
 \leq \int_K \Big|f(\phi^t_N(u))-f(\phi^t(u))\Big| \mu^s_N(du)
 + 
 \left|\int_K f(\phi^t(u))\mu^s_N(du)-\int_K f(\phi^t(u))\mu^s(du)\right|
 \\
 =A+B.
\end{multline*}
As $\mu^s_N\rightharpoonup\mu^s$ weakly over $H^r$ and $f\circ \phi^t\in C_b(H^r)$ for fixed $|t|\leq T$ (Lemma~\ref{lem:phitN_converges_phit}), clearly the second term $B\to 0$.
Because $K$ is compact and according to Lemma~\ref{lem:loc_lipschitz}, we can assume without loss of generality that $f\in C_b(H^r)$ is actually $L$-Lipschitz on $K$ for some $L$.
We can therefore estimate for the $A$ term 
\begin{multline*}
 \int_K \Big|f(\phi^t_N(u))-f(\phi^t(u))\Big| \mu^s_N(du)
 \leq L\int_K \|\phi^t_N(u)-\phi^t(u)\|_{H^r}\mu^s_N(du)\\
 \leq L \sup\limits_{u_0\in K}\|\phi^t_N(u)-\phi^t(u)\|_{H^r}.
\end{multline*}
Our Corollary~\ref{cor:phiN_to_phi_uniform_con_compacts} guarantees that this latter quantity converges to zero for fixed $|t|\leq T$, and the proof is complete.
\end{proof}

We can now fully establish our main result for high regularities.

\begin{proof}[Proof of Theorem~\ref{thmstrong}]
 Recall that, for fixed Brownian coefficients $(a_m)_m$ in \eqref{eq:def_Brownian}, we constructed $\mu^s$ so that $\E_{\mu^s} \mathcal M(u)=\frac{A^0}2$ from Proposition~\ref{prop:N_to_infty}.
 Scaling $(a_m)_m\to (\sqrt{n}a_m)_m$ for an arbitrary $n\in\mathbb N$, all the quantities $A^{r}=\sum_m\lambda_m^r|a_m|^2$ scale as $A^r\to nA^r$ for any fixed $r$ and we can therefore choose the $a_m$ coefficients so that in particular $A^0=n$ for any given $n$.
 According to Proposition~\ref{prop:N_to_infty} the corresponding stationary measure $\mu^{s,n}$ satisfies $\E_{\mu^{s,n}} \,\mathcal M(u)=n/2$.
 Let $\Sigma^{s,n}$ be the corresponding invariant set constructed in Proposition~\ref{prop:Sigma_full_measure_invariant}.
Owing to $\mathcal M(u)=\|u\|^2_{\dot H^{s-\alpha}}+ G(\|u\|_{H^s})\|u\|^2\leq \|u\|_{H^s}^2+G(\|u\|_{H^s})\|u\|_{H^s}^2$, we see that $\Sigma^{s,n}$ contains data of size at least $\|u\|_{H^s}\geq \xi(n/2)$, where $\xi$ is the inverse of the increasing map $\rho\mapsto\rho^2+G(\rho)\rho^2$.

Let us now define
\begin{equation}
\label{eq:def_mustar_Sigmastar}
\mu^{s} =\sum_{n\geq 1}\dfrac{\mu^{s,n}}{2^n}\in\mathfrak p(L^2)
\qqtext{and}
\Sigma^s=\bigcup\limits_{n\geq 1}\Sigma^{s,n}.
\end{equation}
Clearly $\Sigma^{s}$ has full-measure (because $\mu^{s,n}(\Sigma^{s,n})=1$), is contained in $H^s$, and $\phi^t$ is globally well-defined on $\Sigma^s$.
Because $\Sigma^{s,n}$ contains large data $\|u\|_{H^s}\geq \xi(n/2)$ for fixed $n$, clearly $\Sigma^{s}$ contains arbitrarily large data.
Since all the $\mu^{s,n}$ are invariant under $\phi^t$, so is $\mu^{s}$.

From \eqref{eq:int_Mmus_=A0}\eqref{eq:estimate_exponential_Hs_mu}, for fixed $n$ we have that
$$
\int_{L^2}\|u\|^2_{H^s}\mu^{s,n}(du)
\leq C\left(nA^0+n A^\alpha + n A^{\frac{d-1}2}(1+nA^0)\right)\leq Cn^2.
$$
hence
$$
\int_{L^2}\|u\|^2_{H^s}\mu^{s}(du)
= \sum \limits_{n\geq 1}\int_{L^2}\|u\|^2_{H^s}\frac{\mu^{s,n}(du)}{2^n}
\leq C\sum\limits_{n\geq 1}\frac{n^2}{2^n}
<+\infty.
$$
Finally, pick an arbitrary $u_0\in \Sigma^{s}$ and fix $r<s$.
By definition we have $u_0\in\Sigma^{s,n}$ for some $n\in \N$, hence from Proposition~\ref{prop:exists_global_Sigma_r}
$$
\|\phi^t (u_0)\|_{H^{r}}\leq 2 \tilde{G}^{-1} (1+i+\ln (1+|t|)),
\hspace{1.3cm}
\forall\,t\in\R
$$
for some $i$ only depending on $u_0$ and $r$.
With our choice of $G(\rho)=c\rho^2e^{\Lambda\beta \rho^2}$ it is easy to check that $\tilde G(\rho)=\ln(1+ G(\rho))\sim\Lambda\beta\rho^2$ is convex and superlinear at infinity, hence $\zeta(z)=2 \tilde{G}^{-1}(z)\sim 2\sqrt{z/\Lambda\beta}$ is concave and sublinear at infinity.
From there it is a simple exercise to show that
$$
\zeta(1+i+\ln (1+|t|))\leq C_\zeta(1+i)\zeta(1+\ln(1+|t|))
$$
for some structural prefactor $C_\zeta(\cdot)$ depending only on the growth parameters $\Lambda,\beta$, and the proof is complete.
\end{proof}
%
\section{Low regularity}
In this section we consider the case $s\leq d/2$.
As already mentioned, the main obstacle in this case is the fact that the existence time for solutions to the projected equation \eqref{eqproj} is exponentially small as $N\to\infty$, which is reminiscent from the fact that even the local well-posedness for the unprojected equation \eqref{eq} is delicate, to say the least.
As a consequence we will lose any quantitative control on the growth of $H^r$ norms over time.
In order to circumvent this difficulty we still implement a fluctuation-dissipation approach, but with a modified dissipation operator $\mathcal L(u)$ in \eqref{eqstointro}.
Roughly speaking, the new dissipation will allow to gain an extra coercivity and control the exponential nonlinearity (which was previously achieved through a control on the $H^s$ norm and the Sobolev embedding).
For technical reasons, this will give a restriction on $s\leq \alpha+1$.
\subsection{Compactness estimates and stationary measure}
In this section, we take
\begin{equation}
\mathcal{L} (u)= P_N \left(u e^{\beta |u|^2}\right) + \left[(-\Delta )^{s-\alpha}+1\right]u.
\label{eq:def_L_weak_dispersion}
\end{equation}
Compared to our previous choice \eqref{eq:def_dissipation_L_strong_disp} for $s>d/2$, the energy and mass dissipations $\mathcal E(u)=E'(u;\mathcal L(u))$, $\mathcal M(u)=M'(u;\mathcal L(u))$ read now
\begin{multline}
\mathcal{E} (u)
=  \left\langle  P_Nu e^{\beta |u|^2} , (-\Delta)^{\alpha } u \right\rangle 
+ 2\beta\left\|P_N ue^{\beta|u|^2}\right\|^2
\\
+ \|u\|_{\dot H^s}^2+\|u\|_{\dot H^\alpha}^2+ 2\beta\left\langle \left[(-\Delta)^{s-\alpha}+1\right] u, P_N u e^{\beta |u|^2}\right\rangle
\label{eq:M_weak_dispersion}
\end{multline}
and
\begin{equation}
\label{eq:M_coercive_weak_dispersion}
\mathcal M(u)=\int_M|u|^2e^{\beta|u|^2}+\|u\|^2_{ H^{s-\alpha}}.
\end{equation}
To deal with the mixed terms appearing in $\mathcal E(u)$, we use an exponential C\'ordoba-C\'ordoba inequality (Lemma~\ref{cordocordo} in the Appendix with parameter $\gamma=s-\alpha\leq 1$) to control (for $u\in E_N$ hence $(-\Delta u)^\theta u\in E_N$ for all $\theta$)
$$
\left\langle P_N u e^{\beta |u|^2} , (-\Delta)^{\alpha } u \right\rangle
=\left\langle u e^{\beta |u|^2} , (-\Delta)^{\alpha } u \right\rangle
\geq \beta\left\|e^{\beta |u|^2/2}  \right\|_{\dot H^{\alpha }}^2 ,
$$
and
\begin{multline*}
\left\langle\left[ (-\Delta)^{s-\alpha }+1\right]u , P_N u e^{\beta |u|^2}\right\rangle 
= 
\left\langle\left[ (-\Delta)^{s-\alpha }+1\right]u ,  u e^{\beta |u|^2}\right\rangle 
\\
\geq \left\langle (-\Delta)^{s-\alpha }u , u e^{\beta |u|^2}\right\rangle 
\geq
\beta\left\|e^{\beta |u|^2/2}\right\|^2_{\dot H^{s-\alpha}}.
\end{multline*}
Discarding these nonnegative terms in \eqref{eq:M_weak_dispersion}, we obtain the coercivity
\begin{equation}
\label{eq:E_coercive_weak_dispersion}
\mathcal{E} (u) 
\geq
2\beta\left\| P_N  ue^{\beta |u|^2}\right\|^2 + \|u\|_{ \dot H^s}^2.
\end{equation}
\begin{rmq}
We stress that in the C\'ordoba-C\'ordoba inequality it is crucial that the exponent $\gamma\leq 1$, and this is the only reason why we had to restrict to $s-\alpha\leq 1$ in this section.
For $\alpha+1<s\leq d/2$ we did not manage to control these two crossed terms, at least not for this choice of dissipation $\mathcal L$, and the functionals $\mathcal M,\mathcal E$ do not provide sufficient control on the nonlinearity for the subsequent analysis to carry over.
\end{rmq}
At this point, and just as before for highly regular $s>d/2$, one should establish the stochastic well-posedness of
\begin{equation}
\partial_t u =- i \left[(-\Delta)^{\alpha}  u + P_N \left(2 \beta   u e^{\beta |u|^2} \right)\right] -\sigma^2 \mathcal{L}(u)+\sigma \eta_N
\label{eqsto_weak}
\end{equation}
in the sense of Definition~\ref{def:stochastic_wellposedness}.
Just as in section~\ref{subsec:FD} and Proposition~\ref{prop:stochastic_wellposedness}, we construct the unique solution $u=v+z$, with $z$ still given by \eqref{eq:def_z} and containing all the Brownian fluctuations.
Similarly to \eqref{eqsto2}, the corrector $v$ should now solve, for any realization of $u_0=u_0^\omega$ and $z(t)=z^\omega(t)$, the deterministic problem
\begin{equation}
\label{eqsto2bis}
\begin{cases}
\frac {dv}{dt} = -i \left((-\Delta)^{\alpha } v + 2\beta(v+z) e^{\beta |v+z|^2}\right)\\
\\ \hspace{2cm} -\sigma^2 \left[P_N\left( (v+z) e^{\beta |v+z|^2}\right) +\big[(-\Delta)^{s-\alpha }+1\big] (v+z)\right],\\
v|_{t=0}=u_0\in E_N.
\end{cases}
\end{equation}
Standard Cauchy-Lipschitz arguments in the finite-dimensional space $E_N$ guarantee local well-posedness, the challenge here is rather proving that solutions are actually global-in-time.
We have
\begin{multline*}
\frac{d}{dt}\frac 12\|v\|^2
= \left\langle v ,  -2i\beta(v+z) e^{\beta |v+z|^2}\right\rangle
\\
\qquad - \sigma^2 \left[ \left\langle v , (v+z) e^{\beta |v+z|^2}\right\rangle
+\left\langle v, \big[(-\Delta)^{s-\alpha}+1\big]  (v+z)\right\rangle \right]
\end{multline*}
Using as before that $\left\langle u,iue^{\beta|u|^2}\right\rangle=0$ (with $u=v+z$) in order to substitute $-z$ for $v$ in the very first term, this gives
\begin{align*}
\frac{d}{dt}\frac 12\|v\|^2
& =
\left\langle -z ,  -2i\beta(v+z) e^{\beta |v+z|^2}\right\rangle
- \sigma^2 \left\langle (v+z)-z , (v+z) e^{\beta |v+z|^2}\right\rangle
\\
& \hspace{1.5cm} -\sigma^2 
 \left\langle v, \big[(-\Delta)^{s-\alpha}+1\big]  (v+z)\right\rangle 
\\
&\leq 
2\beta \int_M|z||v+z|e^{\beta|v+z|^2} -\sigma^2\int_M|v+z|^2e^{\beta|v+z|^2} +\sigma^2\int_M|z||v+z|e^{\beta|v+z|^2}
\\
& \hspace{1.5cm} +\sigma^2 
\Big(-\|v\|^2_{H ^{s-\alpha}}+ \|v\|_{ H ^{s-\alpha}}\|z\|_{ H ^{s-\alpha}}\Big)
\\
&\leq 
\int_M(2\beta+\sigma^2) |z|\times|v+z|e^{\beta|v+z|^2} -\sigma^2\int_M|v+z|^2e^{\beta|v+z|^2}
\\
& \hspace{1.5cm}+\frac {\sigma^2 }2
\Big(-\|v\|^2_{ H ^{s-\alpha}}+\|z\|^2_{ H ^{s-\alpha}}\Big)
\end{align*}
In order to fully exploit the coercivity from the second $-\sigma^2$ term in the right-hand side, let $\Phi_\sigma$ be the convex function given by Lemma~\ref{lem:Phi_generalized_Young} with $b=\beta$ and $c=\sigma^2$.
Applying the generalized Young inequality (with $\Phi^\ast_\sigma$ the Legendre transform of $\Phi_\sigma$), we estimate
\begin{multline*}
\int_M(2\beta+\sigma^2) |z|\times|v+z|e^{\beta|v+z|^2}
\leq
\int_M(2\beta+1) |z|\times|v+z|e^{\beta|v+z|^2}
\leq\\
\int_M\Phi_\sigma^*((2\beta+1) |z|) +
\int_M\Phi_\sigma\left(|v+z|e^{\beta|v+z|^2}\right)
\\
=
\int_M\Phi_\sigma^*((2\beta+1) |z|)
+\sigma^2 \int_M|v+z|^2e^{\beta|v+z|^2}.
\end{multline*}
Recalling from \eqref{eq:bound_z_realization} that $\sup\limits_{\tau\leq t}\|z^{\omega }(\tau)\|^2\leq \sigma^2 C_{\omega,t}$ for fixed $\omega\in\Omega$, we have $\|z^\omega(\tau)\|^2_{L^\infty(M)}\leq C_N \|z^\omega(t)\|^2\leq C_{N,\omega,t} \sigma^2$ for $\tau\leq t$ and we conclude that
\begin{align*}
\frac{d}{d\tau}\frac 12\|v\|^2
&\leq 
\left[\int_M\Phi_\sigma^*((2\beta+1) |z|)
+\sigma^2 \int_M|v+z|^2e^{\beta|v+z|^2}\right] -\sigma^2\int_M|v+z|^2e^{\beta|v+z|^2}
\\
& \hspace{1.5cm} +\frac {\sigma^2 }2
\Big(-\|v\|^2_{ H ^{s-\alpha}}+\|z\|^2_{ H ^{s-\alpha}}\Big)
\\
& \leq Vol(M) \Phi_\sigma^*(C_{\beta,\omega,N,t}\sigma)+\frac {\sigma^2 }2
\Big(-\|v\|^2_{ H ^{s-\alpha}}+\|z\|^2_{ H ^{s-\alpha}}\Big).
\end{align*}
Then either $\| v(t)\|^2_{H ^{s-\alpha}}(t)\leq \| z(t)\|^2_{H ^{s-\alpha}}$, or $\frac{d}{dt}\frac 12\| v\|^2 \leq  Vol(M) \Phi_\sigma^*(C_{\beta,\omega,N,t}\sigma)$.
Since $ z$ is bounded (for fixed $\omega$) both scenarios lead to some quantitative estimate of the form
\begin{equation}
\label{eq:bound_v_weak_dispersion}
\|v^\omega(t)\|^2\leq \|v^\omega(0)\|^2+\sup\limits_{\tau \leq t}\|z^\omega(t)\|^2 + Vol(M)\Phi^*_\sigma(C_{\beta,N,\omega,t}\sigma )t
\end{equation}
preventing blow-up in finite time, and we conclude that $v$ is indeed global-in-time for a.a. $\omega\in \Omega$.
Being the right-hand side of \eqref{eqsto2bis} a locally Lipschitz function of $v$, this a priori bound leads to continuity with respect to deterministic initial data, hence \eqref{eqsto2bis} is again stochastically globally well-posed in the sense of Definition~\ref{def:stochastic_wellposedness}.
\\

Exactly as in  the high regularity case, one can establish at this stage the equivalent of Proposition~\ref{consalphamass} and Proposition~\ref{consalphaener} (adapted to the new dissipation operator $\mathcal L(u)$ in \eqref{eq:def_L_weak_dispersion}).
The coercivity \eqref{eq:M_coercive_weak_dispersion}\eqref{eq:E_coercive_weak_dispersion} gives then enough compactness on the one-parameter family
$$
\lambda_t(\bullet)=\frac 1t\int_0^t \mathfrak{B}_{\sigma ,N}^{\tau\ast} \delta_0 (\bullet) d\tau,
\hspace{2cm}t\geq 1
$$
and just like in the proof of Theorem~\ref{thm:exists_mu_s_sigma_N} we conclude from a Bogoliubov-Krylov argument that 
\begin{center}
 There exists a stationary measure $\mu^s_{\sigma,N}$ for \eqref{eqsto_weak}.
\end{center}

The next step is to take the inviscid limit $\sigma\to 0$.
This can be done here as well, and it is easy to show as in Section~\ref{subsec:FD} that there exists $\mu^s_N\in \mathfrak p(L^2)$ such that, for a discrete sequence $\sigma_k\to 0$
$$
\lim\limits_{k\to+\infty} \mu^s_{\sigma_k ,N}=\mu_N^s
\qquad\mbox{weakly over }L^2.
$$
Note however that, in order to prove that $\mu^s_N$ is invariant for the deterministic flow $\phi_N^t$ in Proposition~\ref{propmainN}, we strongly exploited in \eqref{eq:ODE_wk} the fact that the bound \eqref{eq:a_priori_bound_v} was uniform in $\sigma$.
For low regularities our corresponding bound \eqref{eq:bound_v_weak_dispersion} is unfortunately \emph{not} uniform in $\sigma$, and in fact blows-up as $\sigma\to 0$ due to $\Phi^*_\sigma$ being the Legendre transform of $\Phi_\sigma=\sigma^2\Phi$ for a fixed $\Phi$ (and as a consequence $\Phi_\sigma^*(z)=\sigma^2\Phi^*(z/\sigma^2)$, which indeed blows-up because $\Phi^*$ grows super-exponentially at infinity but $z$ is only controlled as $\mathcal O(\sigma)$).

Fortunately, it suffices to prove in fact that $u_k=v_k+z_k$ in \eqref{eq:ODE_wk} remains  bounded uniformly in $\sigma_k\to 0$ for $\omega\in S_\rho$ only (and not \emph{all} $\omega\in\Omega$ as in the proof of Proposition~\ref{propmainN}).
The key idea for this is that, by definition \eqref{eq:def_Sr} of $S_\rho$, the martingale term is controlled uniformly in $\omega,\sigma$.
More precisely, let $u_k=u_k(t;P_Nu_0)$ be the solutions of the SDE \eqref{eqsto_weak} with dissipation $\mathcal L(u)$ given by \eqref{eq:def_L_weak_dispersion}.
As in the proof of Proposition~\ref{consalphamass}, $M=\frac 12\|u_k\|^2$ has It\^o differential
$$
dM=\sigma^2_k\left(\frac{A^0_N}{2}-\mathcal M(u_k)\right)dt+\sigma_k \sum_{|m|\leq N} a_m (u_k,e_m) d\beta_m.
$$
Because $\mathcal M\geq 0$, and since by definition \eqref{eq:def_Sr} of $S_\rho$ the martingale term is bounded by $\rho\sigma_k t$, we see that
$$
\frac{1}{2}\|u_k(\tau)\|^2=M(\tau)
\leq M(0)+\sigma_k^2\frac{A^0_N}{2}+\rho\sigma_k t
\leq \frac{R^2}2+\sigma_k^2\frac{A_{0}}{2}+\rho\sigma_k t
\leq C_{R,t,\rho}
$$
for all $\tau\leq t$, uniformly in $\sigma_k\to 0$ and $\omega\in S_\rho$.
The rest of the proof of Proposition~\ref{propmainN} remains identical, and we conclude as before that
\begin{center}
 $\mu^s_N=\lim\limits_{k\to\infty}\mu^s_{\sigma_k,N}$ is invariant for $\phi^t_N$.
\end{center}
Moreover, and omitting the details for the sake of exposition, the coercivity \eqref{eq:M_coercive_weak_dispersion}\eqref{eq:E_coercive_weak_dispersion} of $\mathcal M,\mathcal E$ give altogether
\begin{equation}
\label{boundE2}
\int_{L^2} \left(\|u\|^2_{\dot H^s} +\left\|P_N u e^{\beta |u|^2}\right\|^2 +\int_M|u|^2e^{\beta|u|^2}\right) \mu_N^s (du)\leq C
\end{equation}
for some $C>0$ independent of $N$.
Compared to the high regularity in section~\ref{sec:strong}, we lost quantitative estimates on the growth-rate of $H^r$ norms, but we gained in exchange $L^2$ control over the projected nonlinearity.
%
%
\subsection{Construction of global solutions}
The goal is now to pass to the limit $N\to\infty$, and this is where the Skorokhod-type argument \cite{albeverio1990global,BTT,da2002two} comes into play.
To this end, let $u_{0}$ be an $E_N$-valued random variable distributed according to $\mu^s_N\in \mathfrak p(L^2)$ just constructed, and let $u_N(t)=\phi^t_N(u_0)$.
As a first step we work in a finite time interval $t\in [-T,T]$ for some fixed large $T>0$.
One can canonically consider the processes $(u_N (t))_{t\in [-T,T]}$ as a $C([-T,T],H^{s-\alpha})$-valued random variable.
Let $\nu^s_{T,N}$ be its law.
By invariance of $\mu^s_N$ we have
\begin{equation}
\label{rellaw}
\mu^s_N =\left.\nu^s_{T,N} \right\vert_{t=t_0}
\qqtext{for all}
t_0 \in [-T,T].
\end{equation}
Consider the following functional spaces
$$
X_T^s = L^2 ([-T,T] ;H^{s}) \cap  H^1 \Big([-T,T]; \big(\dot H^{s-\alpha}+L^2)\big)\Big)
$$
and
$$
Y_T^{s}= C([-T,T] ;H^{s-\alpha}),
$$
and note that $X_T^s$ is compactly embedded into $Y_T^s$.
\begin{prop}
We have 
\begin{equation}
\int_{Y^s_T}\left( \|u\|_{X_T^s}^2+\left\|ue^{\beta|u|^2}\right\|_{L^1_tL^1_x}\right) \nu_{T,N}^s (du) \leq C T
\label{eq:estimate_nu_kNs}
    \end{equation}
for some constant $C$ independent of $N$.
\end{prop}

\begin{proof}
From $-i\partial_t u_N=(-\Delta)^\alpha u_N +2\beta P_N u_Ne^{\beta|u_N|^2}$ we have obviously
$$
\|\partial_t u_N\|^2_{\dot H^{s-\alpha}+L^2}
\leq \|(-\Delta)^\alpha u_N\|^2_{\dot H^{s-\alpha}} + \left\|2\beta P_N u_Ne^{\beta|u_N|^2}\right\|^2_{L^2},
$$
hence by \eqref{rellaw} and the invariance of $\mu^s_N$
\begin{multline*}
 \int_{Y^s_T} \|u\|_{X_T}^2  \nu_{T,N}^s (du) =
 \int_{-T}^T \int_{L^2}\left(\|u_N(t)\|^2_{H^s}+\|\partial_t u_N(t)\|^2_{\dot H^{s-\alpha}+L^2}\right)\nu^{s}_{T,N}\big|_{t}(du) dt
 \\
 \leq
 \int_{-T}^T \int_{L^2}\left(\|u_N(t)\|^2_{H^s}+\|(-\Delta)^\alpha u_N(t)\|^2_{\dot H^{s-\alpha}} + 4\beta^2\left\|P_N u_N(t)e^{\beta|u_N(t)|^2}\right\|^2_{H^s}\right)\nu^{s}_{T,N}\big|_{t}(du) dt
 \\
 =
 \int_{-T}^T \int_{L^2}\left(\|u\|^2_{H^s}+\|(-\Delta)^\alpha u\|^2_{\dot H^{s-\alpha}} + 4\beta^2\left\|P_N ue^{\beta|u(t)|^2}\right\|^2_{H^s}\right)\mu^{s}_{N}(du) dt
 \\
 \leq (1+4\beta^2)2T\int_{L^2}\left(\|u\|^2_{H^s}+\|u\|^2_{\dot H^{s}} + \left\|P_N ue^{\beta|u|^2}\right\|^2_{L^2}\right)\mu^{s}_{N}(du).
\end{multline*}
By \eqref{boundE2} this latter integral is bounded uniformly in $N$ and the first part of our estimate follows.

In order to control now the $L^1_tL^1_x$ norm of the nonlinearity, we first argue just as before to write
$$
\int_{Y^s_T} \left\|ue^{\beta|u|^2}\right\|_{L^1_tL^1_x} \nu_{T,N}^s (du)
=
2T\int_{L^2}\int_M\left|ue^{\beta|u|^2}\right|dx\,\mu_{N}^s(du).
$$
Let now $\Phi(z)$ be the monotone increasing, convex function from Lemma~\ref{lem:Phi_generalized_Young} with $b=\beta$ and $c=1$, defined implicitly as $\Phi(ue^{\beta|u|^2})=u^2e^{\beta|u|^2}$.
By Jensen's inequality and \eqref{boundE2} we have
\begin{multline*}
\Phi\left(\int_{L^2}\int_M\left|ue^{\beta|u|^2}\right|\frac{dx}{Vol(M)}\,\mu_{N}^s(du)\right)
\leq \int_{L^2}\int_M\Phi\left(ue^{\beta|u|^2}\right)\frac{dx}{Vol(M)}\,\mu_{N}^s(du)
\\
=\int_{L^2}\int_M|u|^2e^{\beta|u|^2}\frac{dx}{Vol(M)}\,\mu_{N}^s(du)
\leq C
\end{multline*}
uniformly in $N$, hence
\begin{equation}
\int_{Y^s_T} \left\|ue^{\beta|u|^2}\right\|_{L^1_TL^1_x} \nu_{T,N}^s (du)\leq 2T\, Vol(M)\Phi^{-1}(C)
\label{eq:estimate_ue_beta_u2_N}
\end{equation}
and the proof is complete
\end{proof}
Chebyshev's inequality in \eqref{eq:estimate_nu_kNs} gives $\nu^s_{T,N}(B_R(X^s_T)^c)\leq \frac{C}{R^2}$ uniformly in $N$, hence by the Prokhorov theorem there exists a weak limit
$$
\nu^{s}_T=\lim\limits_{N\to\infty}\nu^{s}_{T,N}
\qqtext{in}
\mathfrak p(Y_T^s)
$$
(up to a subsequence $N_k\to\infty$, not relabeled here.)
By lower semicontinuity in \eqref{eq:estimate_nu_kNs} we have moreover
\begin{equation}
\label{eq:control_u_nu_Ts}
\int_{Y^s_T}\left( \|u\|_{X_T^s}^2+\left\|ue^{\beta|u|^2}\right\|_{L^1_tL^1_x}\right) \nu_T^s (du) \leq C T
\end{equation}
By Skorokhod's representation theorem, there exists a $Y_T$-valued stochastic process $u$ distributed according to $\nu^{s}_T\in\mathfrak p(Y_T^s)$ and such that 
$$
u_N\to u\qtext{ in }Y_T^{s}=C([-T,T];H^{s-\alpha}), \qquad \mathbb{P}-\mbox{almost surely}.
$$
Next, we want to show that $u$ is a distributional solution of \eqref{eq} for $\mathbb P$-a.a. $\omega$.
By construction $u_N$ is a global distributional solution of \eqref{eqproj}, and the above almost sure convergence in $C([0,T];H^{s-\alpha})$ implies the distributional convergence of the linear terms $i\partial_t u_N-(-\Delta)^\alpha u_N$.
Hence it clearly suffices to prove that the nonlinearity also passes to the limit in the sense of distributions,
$$
F_N(u_N)=P_N\left(u_Ne^{\beta|u_N|^2}\right) \xto{\mathcal D'_{tx}} ue^{\beta|u|^2}=F(u),
\hspace{1.5cm}
\mathbb P-\mbox{almost surely}.
$$
To this end we decompose, for some arbitrarily large $M\in \N$,
\begin{multline}
\label{convbad}
F_N (u_N)-F(u)
= [F_N (u_N) - F(u_N)]+ [F(u_N) - F_M (u_N)]
\\+[F_M (u_N) - F_M (u)] +[F_M (u) - F(u)].
\end{multline}
By the almost sure convergence of $u_N$ to $u$ and the continuity of $F_M$, we see that the third term of the right-hand side converges almost surely to $0$ at least in $C ([-T,T] ; H^{s-\alpha})$.
Let us now deal with the first term $F_N(u_N)-F_N(u)$, and fix some large $r>0$.
Write for simplicity
$$
v_N=u_Ne^{\beta|u_N|^2}
\qqtext{and}
v=ue^{\beta|u|^2},
$$
so that $F_N(u_N)=P_N v_N$ and $F(u_N)=v_N$.
By \eqref{eq:estimate_ue_beta_u2_N} we have, for almost-all $\omega\in \Omega$, that
$$
\|v_N\|_{L^1_tL^1_x}\leq C_\omega
$$
uniformly in $N$.
If the Sobolev exponent $-r<0$ is sufficiently negative, then there exists $\eps=\eps(r,d)>0$ such that $\|(1-P_N)f\|_{H^{-r}}\leq CN^{-\eps}\|f\|_{L^1}$ for all $f\in L^1(M)$.
As a consequence the first term in \eqref{convbad} can be controlled as
\begin{multline*}
 \left\|F_N (u_N) - F(u_N)\right\|_{L^1_tH^{-r}_x}
 =\left\|(1-P_N)v_N\right\|_{L^1_tH^{-r}_x}
 \leq N^{-\eps}\left\|v_N\right\|_{L^1_tL^1_x}
 \leq N^{-\eps}C_{\omega},
\end{multline*}
and in particular $F_N(u_N)-F(u_N)\to 0$ in the distributional sense.
The remaining terms $F(u_N)-F_M(u_N)$ and $F_M(u)-F(u)$ in \eqref{convbad} can be treated similarly, leveraging now a $M^{-\eps}$ decay rate for arbitrary $M$ (here one should also exploit from \eqref{eq:control_u_nu_Ts} the information that $\|ue^{\beta|u|^2}\|_{L^1_tL^1_x}\leq C_\omega<\infty$ for a.a. $\omega$).
This shows as desired that $u$ is a distributional solution on $[-T,T]$.
\\

In order to conclude the proof of Theorem~\ref{thmweak} in the case $s\leq 1+\alpha$, observe that $u=\lim u_N$ is a solution in the fixed, arbitrarily large time-interval $[-T,T]$.
By a straightforward diagonal extraction argument, it is easy to recover as $T\to\infty$ a limiting $C(\R;H^{s-\alpha})$-valued stochastic process, still denoted $u=\lim u_N$, that solves \eqref{eq} globally in time.
Denoting $\nu^s$ its law and passing to the limit in \eqref{rellaw}, we see that the law $\nu^s\big|_{t=t_0}$ of $u(t_0)$ at an arbitrary time $t_0\in\R$ is given by $\mu^s=\lim \mu^s_N$.
This means exactly that $\mu^s$ is invariant.
Finally, the first term in \eqref{boundE2} immediately passes to the limit as $\int _{L^2}\|u\|^2_{H^s}\mu^s(du)\leq \liminf \int _{L^2}\|u\|^2_{H^s}\mu^s_N(du)\leq C<\infty$.

For the requirement that $\mu^{s}$ contains large initial data in Theorem~\ref{thmweak}, we argue exactly as in Section~\ref{sec:variations}: we first retrieve $\int_{L^2} \mathcal E(u)\mu^s_N(du)=\frac{A^0_N}{2}$ as in Proposition~\ref{propmainN}, and then pass to the limit $N\to\infty$ as in Proposition~\ref{prop:N_to_infty}to retrieve
$$
\int_{L^2}\|u\|^2_{\dot H^s}\mu^s(du)\leq \int_{L^2} \mathcal E(u)\mu^s(du)=\frac{A_{0}}{2}=\frac{1}{2}\sum_{m\in\Z} |a_m|^2.
$$
Scaling the Brownian coefficients $a_m\rightarrow na_m$ for $n\in\N$ gives a corresponding measure $\mu^{s,n}\in\mathfrak p(L^2)$ charging data of size $\|u\|_{\dot H^s}\approx n$.
The exact same construction \eqref{eq:def_mustar_Sigmastar} as before ultimately gives an invariant measure $\mu^{s}=\sum_n\frac{\mu^{s,n}}{2^n}$ charging arbitrarily large data, in particular $\mu^{s}$ is not trivial and the proof of Theorem~\ref{thmweak} is finally complete for $s\leq \alpha+1$.


\begin{appendices}

\section{Some technical results}

\begin{lem}
Let $s>d/2$.
There exists a constant $C=C(s,d)$ such that
$$
\left\|e^{\beta |u|^2}\right\|_{H^s} \leq e^{C\beta \|u\|_{H^s}^2}
\qqtext{and}
\left\|ue^{\beta|u|^2}\right\|_{H^s} \leq C \|u\|_{H^s} e^{C \beta\|u\|_{H^s}^2}
$$
and
$$
\left\|ue^{\beta|u|^2}-ve^{\beta|v|^2}\right\|_{H^s}
\leq Ce^{C\beta\left(\|u\|^2_{H^s}+\|v\|^2_{H^s}\right)} \|u-v\|_{H^s}.
$$
\label{lemsob}
\end{lem}
\begin{proof}
First, let us recall that, for all $u,v\in H^s$, we have
\begin{equation}
\|uv\|_{H^s}
\leq C (\|u\|_{H^s} \|v\|_{L^\infty} +\|u\|_{L^\infty} \|v\|_{H^s})
\leq C \|u\|_{H^s} \|v\|_{H^s}
\label{eq:estimate_uv_Hs}
\end{equation}
for some $C=C(s)$.
Using the previous inequality repeatedly, we get
\begin{multline*}
\left\|e^{\beta|u|^2}\right\|_{H^s} =
\left\|\sum_{k\geq 0} \dfrac{\beta^k|u|^{2k}}{k!} \right\|_{H^s}
\leq
\sum_{k\geq 0} \dfrac{\beta^k}{k!} \left\||u|^{2k}\right\|_{H^s}
\\
\leq \sum_{k\geq 0} \dfrac{\beta^k}{k!} \left(C\|u\|_{H^s}\right)^{2k}
=e^{C\beta \|u\|_{H^s}^2},
\end{multline*}
and as a consequence we obtain as well
$$
\left\|ue^{\beta|u|^2}\right\|_{H^s}
\leq C\|u\|_{H^s}\left\|e^{\beta|u|^2}\right\|_{H^s}
\leq C\|u\|_{H^s}e^{C\beta\|u\|^2_{H^s}}.
$$
In order to estimate $ue^{\beta|u|^2}-ve^{\beta|v|^2}$ we first exploit the algebraic identity
$|u|^{2k}u-|v|^{2k}v=(u-v)P_k(u,\bar u,v,\bar v)$, where $P_k$ is a homogeneous polynomial of degree $2k$ defined recursively by $P_{k}=|u|^{2k}+|v|^{2k}+uv P_{k-1}$ and $P_0=1$.
It is not difficult to check that $\|P_k\|_{H^s}\leq C^k (\|u\|^2_{H^s}+\|v\|^2_{H^s})^k$, hence
\begin{multline*}
 \left\|ue^{\beta|u|^2}-ve^{\beta|v|^2}\right\|_{H^s}
 =
 \left\|\sum\limits_{k\geq 0}\frac{\beta^{k}}{k!}(|u|^{2k}u-|v|^{2k}v)\right\|_{H^s}
 \\
=\left\|\sum\limits_{k\geq 0}\frac{\beta^{k}}{k!}(u-v)P_k\right\|_{H^s}
 \leq \sum\limits_{k\geq 0}\frac{\beta^k}{k!}C\|(u-v)\|_{H^s}\|P_k\|_{H^s}
 \\
 \leq C\left\|u-v\right\|_{H^s}
 \sum\limits_{k\geq 0}\frac{(\beta C)^k}{k!}(\|u\|^{2}_{H^s}+\|v\|^{2}_{H^s})^k
 =C\left\|u-v\right\|_{H^s} e^{C\beta\left(\left\|u\right\|^2_{H^s}+\left\|v\right\|^2_{H^s}\right)}
\end{multline*}
and the proof is complete.
\end{proof}

\begin{lem}[C\'ordoba-C\'ordoba inequality {\cite{constantin2017remarks,cordoba2003pointwise}}]
Let $\Phi \in C^2 (\R ,\R)$ satisfying $\Phi (0)=0$, and fix $\gamma \in (0,1]$.
For any $f\in C^\infty (M,\R)$ there holds
$$
\Phi' (f)(x) (-\Delta)^\gamma f(x) \geq (-\Delta )^\gamma \Phi (f)(x),
\hspace{1.2cm}
\forall\,x\in M.
$$
\end{lem}

\begin{lem}
\label{cordocordo}
For $\gamma \in (0,1]$ and $f\in C^{\infty}(M,\mathbb C)$, we have
$$
\left\langle f e^{|f|^2}, (-\Delta)^\gamma f \right\rangle
\geq \left\| e^{|f|^2 /2}\right\|^2_{\dot H^\gamma}.
$$
\end{lem}
\begin{proof}
Let $f=a+ib$.
Applying twice the previous lemma with $\Phi(r)=r^2/2$, separately to the real-valued functions $a,b$, we have first
\begin{multline*}
\mathcal{R} \left[(a+ib)e^{(a^2 +b^2)/2} (-\Delta)^\gamma (a-ib) \right]
= e^{(a^2 +b^2)/2} (a (-\Delta)^\gamma a + b (-\Delta)^\gamma b)
\\
\geq e^{(a^2+b^2) /2} \left[(-\Delta )^\gamma (a^2/2)   + (-\Delta)^\gamma (b^2/2)\right]
= e^{(a^2+b^2) /2} (-\Delta )^\gamma \left(\frac{a^2+b^2}2\right).
\end{multline*}
Applying once again the C\'ordoba-C\'ordoba inequality to the real-valued function $g=(a^2 +b^2)/2$ with this time $\Phi(z)=e^z-1$, we find
\begin{multline*}
\mathcal{R} \left[(a+ib)e^{(a^2 +b^2)/2} (-\Delta)^\gamma (a-ib) \right]
\geq \Phi'(g)(-\Delta)^\gamma g\\
\geq (-\Delta)^{\gamma}\Phi(g)
=(-\Delta)^{\gamma}\left( e^{(a^2+b^2)/2}\right)
\end{multline*}
pointwise on $M$.
Finally multiplying by $h=e^{(a^2+b^2)/2}\geq 0$ and integrating by parts in the right-hand side $\int_M h(-\Delta)^\gamma h=\|(-\Delta)^{\gamma/2}h\|^2=\|h\|^2_{\dot H^\gamma}$ gives the result.
\end{proof}
\begin{lem}
\label{lem:Phi_generalized_Young}
 Let $b,c>0$, and set $f(u)=u e^{b u^2}$.
 Then the functional equation
 $$
 \Phi(f(u))=c uf(u)
 $$
 admits a unique $C^1$, strictly convex solution  $\Phi:\R^+\to \R^+$ given by
 $$
 \Phi(v)=c v f^{-1}(v).
 $$
\end{lem}
\begin{proof}
Observe that $f:\R^+\to\R^+$ is $C^\infty$, monotone increasing, bijective, and $f'(u)\geq 1$.
As a consequence $f^{-1}:\R^+\to\R^+$ is also $C^\infty$, monotone increasing, and bijective.
 Changing variables $v=f(u)\Leftrightarrow u =f^{-1}(v)$, the functional equation obviously reads
 $$
 \Phi(v)=cu f(u)=c f^{-1}(v)v.
 $$
 The only thing left to check is that $\Phi$ is convex.
 Differentiating twice $\Phi''(v)=c\frac{d^2}{dv^2}(f^{-1}(v)v)$, a straightforward but tedious computation leads ultimately to
 $$
 \Phi''(v)=\frac{c}{f'(u)}\times\frac{2+2b u^2+4b^2 u^2}{(1+2b u^2)^2}> 0
 $$
 with $u=f^{-1}(v)\geq 0$.
\end{proof}

\begin{lem}[{\cite[\S 5.1.1]{AGS}}]
\label{lem:integral_functional_LSC}
 Let $(X,\sfd)$ be a Polish space and $f:X\to \R\cup\{+\infty\}$ be a lower semicontinuous function bounded from below.
 Then the functional
 $$
 \lambda\in \mathfrak p(X)\mapsto \int_X f(x)\lambda(dx)
 $$
 is lower semi-continuous for the weak convergence of probability measures.
\end{lem}

\begin{lem}[Bogoliubov-Krylov, {\cite[Lemma $B.1$]{SY1}}]
\label{lem:BK}
 Let $(P_t)_{t\geq 0}$ be a Feller semi-group on a Banach space $X$ and denote by $P_t^\ast$ its adjoint. If there exists $t_n \rightarrow \infty$ and $\mu \in \mathfrak{p}(X)$ such that $\frac{1}{t_n}\int_0^{t_n} P_t^\ast \delta_0 dt \rightharpoonup  \mu$ in $X$ then $P_t^\ast \mu =\mu$ for all $t\geq 0$.
\end{lem}

\begin{lem}[\cite{gutev2020lipschitz,miculescu2000approximation}]
\label{lem:loc_lipschitz}
 Let $(X,d)$ be a compact metric space.
 Any $f\in C(X;\R)$ can be uniformly approximated by a sequence of globally Lipschitz functions.
\end{lem}

\begin{lem}[modified Arzel\`a-Ascoli]
\label{lem:pointwise_to_uniform_on_compacts}
 Let $(X,\sfd)$ be a Polish space, $K\subset X$ a compact set, and $f_n:K\subset X\to X$ an equicontinuous sequence of functions converging pointwise
 $$
 f_n(x)\xto{n\to\infty} f(x)\qquad\mbox{for all }x\in K.
 $$
 Then
 $$
 \sup\limits_{x\in K}\sfd(f_n( x),f(x))\xto{n\to\infty} 0.
 $$
\end{lem}
\begin{proof}
 Fix $\eps>0$.
 By equicontinuity there is $\delta>0$ such that $\sfd(f_n(x),f_n(y))+\sfd(f(x),f(y))\leq \eps/2$ as soon as $\sfd(x,y)\leq \delta$.
 Take a finite cover $K\subset \cup_i B(x_i,\delta)$ of radius $\delta$.
 For $x\in K$ we have by triangle inequality
 \begin{multline*}
 \sfd(f_n(x),f(x))\leq 
 \sfd(f_n(x),f_n(x_i))+\sfd(f(x),f(x_i)) + \sfd(f_n(x_i),f(x_i))
 \\
 \leq \eps/2 + \sfd(f_n(x_i),f(x_i))
 \end{multline*}
 for some suitable $i$ such that $x\in B(x_i,\delta)$, and the claim follows.
\end{proof}

\end{appendices}

\bibliographystyle{alpha}
\bibliography{bibFLMT}

\end{document}